\documentclass{amsart}
\usepackage{latexsym,amsmath,amsthm,verbatim,ifthen,amssymb,graphicx,color}
\usepackage[all]{xy}
\usepackage{color}

\SelectTips{cm}{}

\newtheorem{theorem}{Theorem}[section]

 \newtheorem{lemma}[theorem]{Lemma}
 \newtheorem{proposition}[theorem]{Proposition}
 \theoremstyle{definition}
 
 \theoremstyle{remark}
 \newtheorem{remark}[theorem]{Remark}
 \newtheorem{example}[theorem]{Example}
 \numberwithin{equation}{section}

\newcommand{\bz}{\mathbb Z}

\DeclareMathOperator{\Hom}{Hom}
\newcommand{\To}{\longrightarrow}
\DeclareMathOperator{\map}{map}
\newcommand{\homo}{\text{Hom}}
 \def\dl{\operatorname{dl}}
 \def\cl{\operatorname{cl}}
 \def\Wl{\operatorname{Wl}}
 \def\bl{\operatorname{bl}}
 
  \def\Aut{\operatorname{Aut}}

\begin{document}
\title[Homotopy transfer and rational models for mapping spaces]
{Homotopy transfer and rational models\\ for mapping spaces}

\author[U. Buijs]{Urtzi Buijs}
\address{Departamento de \'Algebra, Geometr\'{\i}a y Topolog\'{\i}a, Universidad
de M\'alaga, Ap. 59, 29080 M\'alaga, Spain} \email{ubuijs@uma.es}
\author[J.\,J. Guti\'errez]{Javier J.~Guti\'errez}
\address{Radboud Universiteit Nijmegen, Institute for
Mathematics, Astrophysics, and Particle Physics, Heyendaalseweg 135, 6525 AJ
Nijmegen, The Netherlands} \email{j.gutierrez@math.ru.nl}

\thanks{The authors were supported by the MINECO grant MTM2010-15831 and by the
Generalitat de Catalunya as members of the team 2009~SGR~119. The first-named
author was supported by a U\nobreakdash-Mobility Programme grant of the
European Union's Seventh Framework Programme, G.A. no. 246550 and by the
Junta de Andaluc\'{\i}a grant FQM-213. The second-named author was supported
by the NWO (SPI 61-638)} 

\begin{abstract}
By using homotopy transfer techniques in the context of rational homotopy
theory, we show that if $C$ is a coalgebra model of a space $X$, then the
$A_\infty$-coalgebra structure in $H_*(X;\mathbb{Q})\cong H_*(C)$ induced by
the higher Massey coproducts provides the construction of the Quillen minimal
model of~$X$. We also describe an explicit $L_\infty$-structure on the
complex of linear maps $\Hom(H_*(X; \mathbb{Q}), \pi_*(\Omega Y)\otimes
\mathbb{Q})$, where $X$ is a finite nilpotent CW-complex and~$Y$ is a
nilpotent CW-complex of finite type, modeling the rational homotopy type of
the mapping space $\text{map}(X, Y)$. As an application we give conditions on
the source and target in order to detect rational $H$-space structures on the
components.
\end{abstract}

\maketitle

\section{Introduction}
The version up to homotopy of differential graded Lie algebras, called
$L_\infty$-al\-ge\-bras, had its origins in the context of deformation
theory~\cite{SS85} and has been highly used since then in different
geometrical settings~\cite{cs,Kon03}. Recently, the theory of
$L_{\infty}$-algebras has become a useful tool in order to describe rational
homotopy types of spaces~\cite{BFM11, Get09}. In particular, mapping spaces
seem to be well described in these terms~\cite{Ber, Bui09,  BFM, BGM, Laz}.

More concretely, if $C$ is a coalgebra model of a finite nilpotent
CW-complex~$X$ and $L$ is an $L_{\infty}$-model of a finite type rational
CW-complex ~$Y$, then the complex of linear maps $\Hom (C,L)$ admits an
$L_{\infty}$-structure whose geometrical realization is of the rational
homotopy type of $\map (X,Y)$, the space of continuous functions;
see~\cite{Ber, BFM} for details. Similarly, the complex
$\Hom(\overline{C},L)$, where $\overline{C}$ is the kernel of the
augmentation $\varepsilon \colon C\longrightarrow \mathbb{Q}$ is an
$L_\infty$-model for $\map^*(X,Y)$, the space of pointed continuous
functions.

The aim of this paper is to describe an explicit $L_\infty$-structure in the
complex of linear maps $\Hom(H_*(X; \mathbb{Q}),\pi_*(\Omega Y)\otimes
\mathbb{Q})$  that serves as a model for the rational homotopy type of
$\map(X,Y)$.

If $L$ denotes the $L_\infty$-algebra such that $\mathcal{C}^* (L)=(\Lambda
V, d)$ is the Sullivan minimal model of $Y$ (see Section~\ref{prelim} for
details), then we have that $L\cong\pi_*(\Omega Y)\otimes \mathbb{Q}$ as
graded vector spaces. Let $C$ be a coalgebra model of a finite nilpotent
CW-complex~$X$. The strategy will be to use the homotopy retract between $C$
and $H_*(C)\cong H_*(X;\mathbb{Q})$ (denoted by $H$ from now on), defining
the higher Massey coproducts, and the above explicit
$L_\infty$\nobreakdash-struc\-tu\-re on $\Hom(C,L)$ to induce another
homotopy retract between $\Hom(C, L)$ and $\Hom(H,L)$. Then, we apply the
\emph{homotopy transfer theorem}~\cite{Fuk03, KS00, KS01, LV, Mer99} to give
an $L_\infty$-structure on $\Hom(H,L)$ and an explicit formula for it by
means of rooted trees (see Section~\ref{main_section} for further details).

We prove that, indeed, this $L_{\infty}$-structure exhibits $\Hom(H,L)$ as an
$L_\infty$-model for $\map(X,Y)$. Note, however, that a quasi-isomorphism
between two $L_{\infty}$-algebras is not necessarily a quasi-isomorphism
after applying the generalized cochain functor~$\mathcal{C}^{*}$. Therefore,
we cannot deduce directly from the quasi-isomorphism provided by the homotopy
retract that $\Hom(H,L)$ is an $L_{\infty}$-model for the mapping space. The
same argument can be applied by replacing $(C,\Delta)$ and $H$ with
$(\overline{C},\overline{\Delta})$ and $\overline{H}$, respectively, to model
the space of pointed continuous functions $\map^*(X,Y)$.

Moreover, the homotopy retract between $C$ and $H_*(X; \mathbb{Q})$ used
above has its own interest since we show that it provides the construction of
the Quillen minimal model of $X$; see Theorem~\ref{modelo_minimal}.

Finally, using the model for $\map^*(X,Y)$, we give a necessary condition for
the components of mapping spaces to be of the rational homotopy type of an
$H$-space. This problem has been previously considered in~\cite{Bui09, BM08,
FT05}. We prove a variant of the results obtained in these papers in terms of
the cone length ($\cl$), the Whitehead length ($\Wl$) and the bracket length
($\bl$), that does not implicitly assume the coformality of the target space.
Explicitly, we prove that if $\cl(X)=2$ and $\Wl(Y)<\bl(X)$, then all the
components of $\map^*(X,Y)$ are rationally $H$-spaces.

\section{Preliminaries}
\label{prelim} We will rely on known results from rational homotopy theory
for which~\cite{FHT} is a standard reference. We also assume the reader is
aware of the concepts of homotopy operadic algebras being~\cite{LV} an
excellent reference. With the aim of fixing notation we give some definitions
and sketch some results we will need. Every algebraic object considered
throughout the paper is assumed to be a graded vector space over the
rationals.

\subsection{$A_\infty$-coalgebras and $L_{\infty}$-algebras}
\label{sect:prelim} An \emph{$A_{\infty}$-coalgebra} $C$ is a graded vector
space together with a differential graded algebra structure on the tensor
algebra $T^{+}(s^{-1}C)$ on the desuspension of $C$. This is equivalent to
the existence of a family of degree $k-2$ linear maps $\Delta_k\colon
C\longrightarrow C^{\otimes k}$ satisfying the equation
$$
\sum_{k=1}^i\, \sum_{n=0}^{i-k}(-1)^{k+n+kn}({\rm id}^{\otimes i-k-n}\otimes \Delta_k
\otimes {\rm id}^n)\Delta_{i-k+1}=0.
$$
Any differential graded coalgebra $(C,\delta,\Delta)$ is an
$A_{\infty}$-coalgebra with $\Delta_1=\delta$, $\Delta_2=\Delta$ and
$\Delta_k=0$ for $k>2$. We will denote by $\Delta^{(k)}=(\Delta\otimes{\rm
id}\otimes\cdots\otimes {\rm id})\circ\cdots\circ (\Delta\otimes {\rm
id})\circ \Delta$ with $k$ factors, and $\Delta^{(0)}={\rm id}$.

An $A_{\infty}$-coalgebra is \emph{cocommutative} if $\tau\circ\Delta_k=0$
for every $k\ge 1$, where $\tau\colon T(C)\longrightarrow T(C)\otimes T(C)$
denotes the \emph{unshuffle coproduct}, that is,
$$
\tau(a_1\otimes\cdots\otimes a_n)=\sum_{i=1}^n\sum_{\sigma\in S(i, n-i)}
\epsilon_{\sigma}(a_{\sigma(1)}\otimes\cdots\otimes a_{\sigma(i)})\otimes
(a_{\sigma(i+1)}\otimes\cdots\otimes a_{\sigma(n)}),
$$
where $\epsilon_{\sigma}$ is the signature of $\sigma$ and $S(i,n-i)$ denotes
the set of $(i, n-i)$-shuffles, i.e., permutations $\sigma$ of $n$-elements
such that $\sigma(1)<\cdots<\sigma(i)$ and $\sigma(i+1)<\cdots<\sigma(n)$.

Let $(C, \{\Delta_k\})$ and $(C',\{\Delta'_k\})$ be two
$A_{\infty}$-coalgebras. A \emph{morphism of $A_{\infty}$-coal\-ge\-bras}
from $C$ to $C'$ is a morphism $f\colon C\to C'$ compatible with $\Delta_k$
and $\Delta'_k$. An \emph{$A_{\infty}$\nobreakdash-morphism} from $C$ to $C'$
is a morphism $f\colon T^+(s^{-1}C)\to T^+(s^{-1}C')$ of differential graded
algebras. This is equivalent to the existence of a family of degree $k-1$
maps $f^{(k)}\colon C\to C'^{\otimes k}$ satisfying the usual relations
involving $\Delta_k$ and $\Delta_k'$. An $A_{\infty}$-morphism is a
\emph{quasi-isomorphism} if $f^{(1)}$ is a quasi-isomorphism of complexes.

An \emph{$L_{\infty}$-algebra} or \emph{strongly homotopy Lie algebra} is a
graded vector space $L$ together with a differential graded coalgebra
structure on $\Lambda^+ sL$, the cofree graded cocommutative coalgebra
generated by the suspension. The existence of this structure on $\Lambda^+
sL$ is equivalent to the existence of degree $k-2$ linear maps $\ell_k\colon
L^{\otimes k}\to L$, for $k\ge 1$, satisfying the following two conditions:
\begin{itemize}
\item[{\rm (i)}] For any permutation $\sigma$ of $k$ elements,
$$
\ell_k(x_{\sigma(1)},\ldots, x_{\sigma(k)})=\epsilon_{\sigma}\epsilon\ell_k(x_1,\ldots, x_k),
$$
where $\epsilon_{\sigma}$ is the signature of the permutation and
$\epsilon$ is the sign given by the Koszul convention.
\item[{\rm (ii)}] The \emph{generalized Jacobi identity} holds, that is
$$
\sum_{i+j=n+1}\sum_{\sigma\in S(i, n-i)}\epsilon_{\sigma}\epsilon(-1)^{i(j-1)}
\ell_{n-i}(\ell_i(x_{\sigma(1)},\ldots,x_{\sigma(i)}), x_{\sigma(i+1)},\ldots, x_{\sigma(n)})=0,
$$
where $S(i,n-i)$ denotes the set of $(i, n-i)$-shuffles.
\end{itemize}
Every differential graded Lie algebra $(L, \partial)$ is an
$L_{\infty}$-algebra by setting $\ell_1=\partial$, $\ell_2=[- ,-]$ and
$\ell_k=0$ for $k>2$. An $L_{\infty}$-algebra $(L,\{\ell_k\})$ is called
\emph{minimal} if $\ell_1=0$.

Let $(L, \{\ell_k\})$ and $(L',\{\ell'_k\})$ be two $L_{\infty}$-algebras. A
\emph{morphism of $L_{\infty}$-algebras} from $L$ to $L'$ is a morphism
$f\colon L\to L'$ that is compatible with $\ell_k$ and $\ell'_k$. An
\emph{$L_{\infty}$-morphism} from $L$ to $L'$ is a morphism $f\colon
\Lambda^+ sL\to \Lambda^+ sL'$ of differential graded coalgebras.

Any $L_{\infty}$-morphism is completely determined by the projection $\pi
f\colon\Lambda^+ s L\to sL'$ which is the sum of a system of skew-symmetric
morphisms $f^{(k)}\colon L^{\otimes k}\to L'$ of degree $1-k$. The morphisms
$f^{(k)}$ satisfy an infinite sequence of equations involving the brackets
$\ell_k$ and $\ell'_k$; see for example~\cite{Kon03}. In particular, for
$k=1$, we have that $\ell_1'f^{(1)}-f^{(1)}\ell_1=0$. Therefore
$f^{(1)}\colon (L,\ell_1)\to (L',\ell_1')$ is a map of complexes. Any
morphism of  $L_{\infty}$-algebras is an $L_{\infty}$-morphism of
$L_{\infty}$-algebras. As in the case of $A_{\infty}$\nobreakdash-coalgebras,
an $L_{\infty}$-morphism is a \emph{quasi-isomorphism} if $f^{(1)}$ is a
quasi-isomorphism of complexes. The following result can be found
in~\cite[Theorem~4.6]{Kon03}.
\begin{theorem}
Let $L$ and $L'$ be two $L_{\infty}$-algebras. If $f$ is a quasi-isomorphism
of $L_{\infty}$-algebras from $L$ to $L'$, then there exists another
$L_{\infty}$-morphism from $L'$ to $L$ inducing the inverse isomorphism
between homology of complexes $(L,\ell_1)$ and $(L',\ell_1')$.
\end{theorem}

An $L_{\infty}$-algebra $(L,\{\ell_k\})$ is called \emph{linear contractible}
if $\ell_k=0$ for $k\ge 2$ and $H_*(L,\ell_1)=0$. Since the property of being
linear contractible is not invariant under $L_{\infty}$-isomorphisms, we say
that $L$ is \emph{contractible} if $L$ is isomorphic as an
$L_{\infty}$\nobreakdash-algebra to a linear contractible one. As stated in
\cite[Lemma~4.9]{Kon03}, any $L_{\infty}$\nobreakdash-algebra is
$L_{\infty}$\nobreakdash-isomorphic to the direct sum of a minimal
$L_{\infty}$-algebra and a linear contractible one. Following~\cite{Kon03},
one can prove that any quasi-isomorphism between minimal
$L_{\infty}$-algebras is an isomorphism. Thus, the set of equivalence classes
of $L_{\infty}$-algebras up to quasi-isomorphisms can be identified with the
set of equivalence classes of minimal $L_{\infty}$-algebras up to
$L_{\infty}$-isomorphisms.

The \emph{Maurer--Cartan set} of an $L_{\infty}$-algebra $L$ is the set of
elements $z\in L_{-1}$ such that the infinite series
$$
\sum_{k\ge 1}\frac{1}{k!}\ell_k(z,\stackrel{(k)}{\ldots}, z)
$$
is a finite sum and it is equal to zero. We will denote the set of
Maurer--Cartan elements in $L$ by ${\rm MC}(L)$. If $L$ is an
$L_{\infty}$-algebra and $z\in{\rm MC}(L)$, then
$$
\ell_k^z(x_1,\ldots, x_k)=\sum_{i\ge 0}\frac{1}{i!}\ell_{i+k}(z,\stackrel{(i)}{\ldots},z,x_1,\ldots, x_k)
$$
defines a new $L_{\infty}$-structure denoted by $L^z$ whenever the series is
a finite sum (cf.~\cite[Proposition~4.4]{Get09}). The \emph{perturbed and
truncated} $L_{\infty}$-structure on $L$ is the $L_{\infty}$-algebra
$L^{(z)}$, whose underlying graded vector space is given by
$$
\left(L^{(z)}\right)_i=\left\{
\begin{array}{cl}
L_i & \mbox{if $i>0$}, \\
Z_{\ell_1^z}(L_0) & \mbox{if $i=0$}, \\
0 & \mbox{if $i<0$},
\end{array}
\right.
$$
where $Z_{\ell_1^z}(L_0)$ denotes the space of cycles for the differential
$\ell_1^z$, and with the same brackets as $L^z$.
\begin{remark}
If we want to extend the definition of the classical cochain
functor~$\mathcal{C}^*$ \cite[\S 23]{FHT} from the category of differential
graded Lie algebras to the category of $L_\infty$\nobreakdash-algebras, in
order to define a realization functor for $L_\infty$-algebras as the
composition of the cochain functor and Sullivan's realization functor
\cite[\S 17]{FHT}, then we should constrain the category of
$L_\infty$-algebras. For a discussion of the exact technical requirements
needed with definitions and examples we refer to \cite[\S 1] {BM}. Since all
the $L_\infty$-algebras involved in the next sections are of this type of
{\em mild} $L_\infty$-algebras we consider from now on that all
$L_\infty$-algebras are of this kind.
\end{remark}

Then we have that an $L_{\infty}$-algebra structure on $L$ is the same as a
commutative differential graded algebra structure on $\Lambda(sL)^{\sharp}$,
denoted by $\mathcal{C}^*(L)$, where~$\sharp$ stands for the dual vector
space. More explicitly, if $V$ and $sL$ are dual graded vector spaces, then
$\mathcal{C}^*(L)=(\Lambda V, d)$ with $d=\sum_{j\ge 1} d_j$ and
\begin{equation}
\label{pairing}
 \langle d_j v; sx_1\wedge \cdots\wedge s_j\rangle=(-1)^{\epsilon}\langle v; s\ell_j(x_1,\ldots, x_j)\rangle,
\end{equation}
where $\langle -;-\rangle$ is defined as an extension of the pairing induced
by the isomorphism between $V$ and $(sL)^{\sharp}$, $d_jv\subset \Lambda^j V$
and $\epsilon$ is the sign given by the Koszul convention.

Conversely, if $(\Lambda V, d)$ is a commutative differential graded algebra
of finite type, then an $L_{\infty}$-algebra structure on $s^{-1}V^{\sharp}$
is uniquely determined by the condition $(\Lambda V ,d)=\mathcal{C}^*(L)$. 	

\subsection{The homotopy transfer theorem}
In this section we recall how to transfer $A_{\infty}$- and
$L_{\infty}$-structures along homotopy retracts. This will be the main tool
used in Section~\ref{main_section} to describe an explicit
$L_{\infty}$-structure on the complex of linear maps $\Hom(H_*(X;
\mathbb{Q}),\pi_*(\Omega Y)\otimes \mathbb{Q})$.

Let $(A, d_A)$ and $(V, d_V)$ be two complexes. We say that $V$ is a
\emph{homotopy retract} of $A$ if there exist maps
$$
\xymatrix{ \ar@(ul,dl)@<-7ex>[]_h  & (A,d_A )
\ar@<0.75ex>[r]^-p & (V,d_V) \ar@<0.75ex>[l]^-i }
$$
such that ${\rm id}_A-ip=d_A h+hd_A$ and $i$ is a quasi-isomorphism. If $(A,
V, i, p, h)$ is a homotopy retract, then it is possible to transfer
$A_{\infty}$- and $L_{\infty}$-structures from $A$ to $V$ with explicit
formulae. This is in fact a particular instance of the so-called
\emph{homotopy transfer theorem} \cite{Fuk03, KS00, KS01, LV, Mer99}, which
goes back to \cite{GS86, GLS91, HK91, Kai83} for the case of
$A_{\infty}$-structures. Before stating it, we need to introduce some
notation on rooted trees.

Let $T_k$ (respectively $PT_k$) be the set of isomorphism classes of directed
rooted trees (respectively planar rooted trees) with internal vertices of
valence at least two and exactly $k$ leaves. Let $(C, \{\Delta_k\})$ be an
$A_{\infty}$-coalgebra and let $(C,V, i, p, h)$ be a homotopy retract of $C$.
For each planar tree $T$ in $PT_k$, we define a linear map $\Delta_T\colon
V\to V^{\otimes k}$ as follows. The leaves of the tree are labeled by $p$,
each internal edge is labelled by $h$  and the root edge is labelled by $i$.
Every internal vertex $v$ is labelled by the operation $\Delta_r$, where $r$
is the number of input edges of $v$. Moving up from the root to the leaves
one defines $\Delta_T$ as the composition of the different labels. For
example, the tree $T$
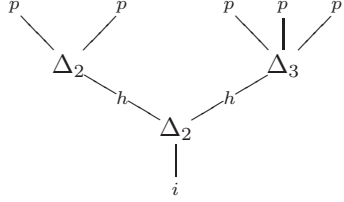
$$
\xymatrixcolsep{1pc}
\xymatrixrowsep{1pc}
\entrymodifiers={=<1pc>} \xymatrix{
*{^p}\ar@{-}[dr] & *{} & *{^p}\ar@{-}[dl] & *{} & *{^p}\ar@{-}[dr] & *{^p}\ar@{-}[d]& *{^p} \ar@{-}[dl]\\
*{} & {\Delta_2} \ar@{-}[drr]|h & *{} & *{} & *{} & \Delta_3\ar@{-}[dll]|h & *{} \\
*{} & *{} & *{} & \Delta_2\ar@{-}[d] & *{} & *{} & *{} \\
*{} & *{} & *{} & *{_{\stackrel{}{i}}} & *{} & *{} & *{} \\
}
$$
yields the map $\Delta_T=(((p\otimes p)\circ \Delta_2\circ
h)\otimes((p\otimes p\otimes p)\circ \Delta_3\circ h))\circ \Delta_2\circ i.
$

Similarly, if $(L,\{\ell_k\})$ is an $L_{\infty}$-algebra and $(L,V, i, p,
h)$ be a homotopy retract of $L$, then each tree $T$ in $T_k$ gives rise to a
linear map $\ell_T\colon \Lambda^k V\to V$ in the following way. Let
$\widetilde T$ be a planar embedding of $T$. If we label the leaves of the
tree by $i$, each internal edge by $h$, the root edge by $p$ and each
internal vertex by $\ell_k$, where $k$ is the number of input edges, then
this planar embedding defines a linear map
$$
\ell_{\widetilde{T}}\colon V^{\otimes k}\longrightarrow V
$$
by moving down from the leaves to the root, according to the usual operadic
rules. For example, for the same tree as before, the labeling reads
$$
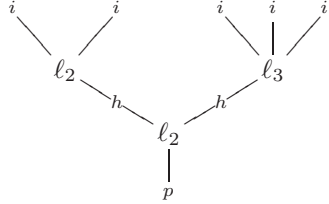
\xymatrixcolsep{1pc}
\xymatrixrowsep{1pc}
\entrymodifiers={=<1pc>} \xymatrix{
*{^i}\ar@{-}[dr] & *{} & *{^i}\ar@{-}[dl] & *{} & *{^i}\ar@{-}[dr] & *{^i}\ar@{-}[d]& *{^i} \ar@{-}[dl]\\
*{} & {\ell_2} \ar@{-}[drr]|h & *{} & *{} & *{} & \ell_3\ar@{-}[dll]|h & *{} \\
*{} & *{} & *{} & \ell_2\ar@{-}[d] & *{} & *{} & *{} \\
*{} & *{} & *{} & *{_{\stackrel{}{p}}} & *{} & *{} & *{} \\
}
$$
and the linear map $\ell_{\widetilde{T}}$ corresponds to
$$
p\circ\ell_2\circ((h\circ \ell_2\circ(i\otimes i))\otimes (h\circ \ell_3\circ (i\otimes i\otimes i))).
$$
Then, we define $\ell_T$ as the composition of $\ell_{\widetilde{T}}$ with
the symmetrization map $\Lambda^k V\to V^{\otimes k}$ given by
$$
v_1\wedge\cdots\wedge v_k\mapsto \sum_{\sigma\in S_k}\epsilon_{\sigma}\epsilon\, v_{\sigma(1)}\otimes\cdots\otimes v_{\sigma(n)},
$$
where $S_k$ denotes the symmetric group on $k$ letters, $\epsilon_{\sigma}$
denotes the signature of the permutation and $\epsilon$ stands for the sign
given by the Koszul convention.

Part (i) of the following theorem follows from the \emph{tensor trick} of
\cite{GLS91} (see also \cite[Section 2]{BR10} for details). The second part
is \cite[Theorem 10.3.9]{LV}.
\begin{theorem}[Homotopy transfer theorem]
Let $(A, d_A)$ and $(V,d_V)$ be two complexes and let $(A, V, i, p, h)$  be a
homotopy retract. Then the following hold:
\begin{itemize}
\item[{\rm (i)}] If $A=(C, \{\Delta_k\})$ is an $A_{\infty}$-coalgebra,
    then there exists an $A_{\infty}$-coalgebra structure $\{\Delta'_k\}$
    on $V$ and an $A_{\infty}$-quasi-isomorphism
$$
i_{\infty}\colon (V,\{\Delta'_k\})\longrightarrow (C,\{\Delta_k\})
$$
such that $\Delta'_1=d_V$ and $i^{(1)}_{\infty}=i$. Moreover, the
transferred higher comultiplications can be explicitly described by the
formula
$$
\Delta'_k=\sum_{T\in PT_k}\Delta_T.
$$
\item[{\rm (ii)}] If $A=(L, \{\ell_k\})$ is an $L_{\infty}$-algebra, then
    there exists an $L_{\infty}$-algebra structure $\{\ell'_k\}$ on $V$
    and an $L_{\infty}$-quasi-isomorphism
$$
i_{\infty}\colon (V,\{\ell'_k\})\longrightarrow (L,\{\ell_k\})
$$
such that $\ell'_1=d_V$ and $i^{(1)}_{\infty}=i$. Moreover, the
transferred higher brackets can be explicitly described by the formula
$$
\ell'_k=\sum_{T\in T_k}\frac{\ell_T}{|\Aut (T)|},
$$
where $\Aut (T)$ is the automorphism group of the tree $T$. $\hfill\qed$
\end{itemize}
\label{htt}
\end{theorem}

\subsection{Rational models for mapping spaces}
\label{sect:rational_models_map} In~\cite{Sul78} Sullivan associated to each
nilpotent space $Z$ a commutative differential graded algebra $A_{PL}(Z)$. In
fact, there is an adjoint pair
$$
\xymatrix{ A_{PL}\colon {\rm sSets^{op}}\ar@<0.75ex>[r] &
{\rm CDGA}\colon {\langle - \rangle}\ar@<0.75ex>[l]},
$$
where ${\rm CDGA}$ is the category of commutative differential graded
algebras, ${\rm sSets}$ is the category of simplicial sets, and $\langle -
\rangle$ denotes the simplicial realization. The \emph{minimal model} of $Z$
is defined as the minimal model $(\Lambda V, d)$ of $A_{PL}(Z)$.
A~\emph{model} of $Z$ is a graded commutative differential algebra
quasi-isomorphic to its minimal model. For more details, we refer
to~\cite{FHT, Sul78}.

By a model of a not necessarily connected space~$Z$, such that all its
components are nilpotent (or a map between them), we mean a
$\mathbb{Z}$-graded commutative differential graded algebra (or a morphism)
whose simplicial realization, in the sense of~\cite{Sul78}, has the same
homotopy type as the singular simplicial approximation of~$Z_{\mathbb{Q}}$.
Similarly, by an $L_{\infty}$-model of a space $Z$ as above, we mean an
$L_{\infty}$-algebra $L$ such that $\mathcal{C}^*(L)$ is a commutative
differential graded algebra model of~$Z$.

\begin{proposition}[\cite{Ber, BS97, BM}]
Let $L$ be an $L_{\infty}$-algebra and $z\in L_{-1}$ a Maurer--Cartan
element. Then there is  a homotopy equivalence
$$
\langle \mathcal{C}^*(L^{(z)})\rangle\stackrel{\simeq}{\longrightarrow}
\langle \mathcal{C}^*(L)\rangle_{z},
$$
where $\langle \mathcal{C}^*(L)\rangle_{z}$ denotes the connected component
containing the $0$-simplex associated to~$z$. $\hfill\qed$
\end{proposition}

Note that this generalizes the notion of differential graded Lie model of a
finite type nilpotent space $Z$, since in this case $\mathcal{C}^*$ coincides
with the classical cochain functor and the only Maurer--Cartan element in $L$
is the zero element.

Similarly, we say that an \emph{$A_{\infty}$-coalgebra model} of $Z$ is a
cocommutative $A_{\infty}$\nobreakdash-coal\-ge\-bra $C$ such that
$\mathcal{L}(C)$ is a differential graded Lie model of $Z$, where
$\mathcal{L}$ denotes the \emph{generalized Quillen functor}; see~\cite{BM}
for further details. This functor assigns to a cocommutative
$A_{\infty}$-coalgebra $C$ an induced differential graded Lie algebra
structure on ${\mathbb{L}}(s^{-1}C)$ whose differential $\partial=\sum_{k\ge
1}\partial_k$ with $\partial_k\colon s^{-1}C\to \mathbb{L}^k(s^{-1}C)$ is
determined by $\Delta_k$ in the same way as the classical Quillen functor
assigns a differential $\partial=\partial_1+\partial_2$ on
$\mathbb{L}(s^{-1}C)$; see, e.g., \cite[IV.22]{FHT}. In fact, if $C$ is a
cocommutative differential graded coalgebra viewed as a cocommutative
$A_{\infty}$\nobreakdash-coalgebra, then $\mathcal{L}$ coincides with the
classical Quillen functor.

In the rest of the paper $X$ will always denote a \emph{nilpotent finite
CW-complex} and $Y$ will always denote a \emph{rational finite type
CW-complex}, although most of the results can be stated if we remove the
finiteness assumption on $X$, as in~\cite{BFM, BFM11}.

We recall briefly the Haefliger model \cite{Hae82} of the mapping space via
the functorial description of Brown\nobreakdash--Szczarba \cite{BS97}. Let
$B$ be a finite dimensional differential graded algebra model of $X$ and let
$(\Lambda V, d)$ be a Sullivan (non-necessarily minimal) model of~$Y$. We
denote by $B^\sharp$ the differential coalgebra dual of $B$ with the grading
$({B^\sharp})^{-n}=B^\sharp_{n}=\Hom(B^n,\mathbb{Q})$ and consider the free
commutative algebra $\Lambda (\Lambda^+ V\otimes {B^\sharp})$ generated by
the $\bz$-graded vector space $\Lambda^+ V\otimes {B^\sharp}$, endowed with
the differential $d$ induced by the ones on $(\Lambda V,d)$ and on
$({B^\sharp}, \delta)$. Let $J$ be the differential ideal generated by
$1\otimes 1-1$, and the elements of the form
$$
v_1v_2\otimes \beta -\sum_j(-1)^{|v_2||\beta_j'|}(v_1\otimes
\beta_j')(v_2\otimes \beta _j''),\quad v_1,v_2\in V,
$$
where
 $\Delta \beta =\sum_j \beta_j'\otimes
\beta_j''$. The inclusion $V\otimes {B^\sharp}\hookrightarrow  \Lambda^+
V\otimes {B^\sharp}$ induces an isomorphism of graded algebras
$$
\rho \colon \Lambda ( V\otimes {B^\sharp})\stackrel{\cong}{\To} \Lambda (\Lambda^+ V\otimes {B^\sharp})/J,
$$
and thus $\widetilde{d}=\rho^{-1}d\rho $ defines a differential in $\Lambda
(V\otimes {B^\sharp})$ and the following holds:

\begin{theorem}[\cite{BS97, Hae82}] The commutative differential graded algebra $(\Lambda
(V\otimes {B^\sharp}), \widetilde{d} )$ is a model of {\em $\map(X,Y)$}, and
the commutative differential graded algebra  $(\Lambda (V\otimes B_+^\sharp
), \widetilde{d} )$ is a model of $\map^* (X,Y)$. $\hfill\qed$ \label{BSH}
\end{theorem}

 Now write $B^\sharp=A\oplus \delta A\oplus H$, where $H\cong
 H(B^\sharp)$,
with basis  $\{a_j\}$, $\{b_j\}$ and $\{h_s \}$. Thus $\delta a_j=b_j$ and
$\delta h_s=0$. Additionally, since $(\Lambda V,d)$ is a Sullivan algebra, we
can choose a basis $\{v_i \}$ for $V$ for which $dv_i\in \Lambda V_{<i}$.
Then we have:

\begin{lemma}[\cite{BS97, Bui09}] The commutative differential graded algebra
$(\Lambda (V\otimes B^\sharp), \widetilde{d})$ splits as $(\Lambda
W,\widetilde{d})\otimes \Lambda (U\oplus\widetilde{d}U)$, where
\begin{itemize}
\item[{\rm (i)}] $U$ is generated by $u_{ij}=v_i\otimes a_j$;
\item[{\rm (ii)}] $W$ is generated by $w_{is}=v_i\otimes h_s-x_{is}$, for
    suitable $x_{is}\in \Lambda (V_{<i}\otimes B^\sharp)$;
\item[{\rm (iii)}] $\widetilde{d}w_{is} \in \Lambda \{ w_{ms}\}_{m<i}$;
\item[{\rm (iv)}] if $\widetilde{d}(v_i\otimes h_s)$ is decomposable, so
    is $\widetilde{d}w_{is}$. \label{splitting_lemma}
\end{itemize}
\end{lemma}
\begin{proof}
We proceed by induction on $i$. Suppose that $w_{ms}$ has been defined for
$m<i$ satisfying the lemma for $(\Lambda (V_{<i}\otimes B^\sharp),
\widetilde{d})$. Now, since $\widetilde{d}(v_i\otimes
h_s)=\rho^{-1}[dv_i\otimes h_s]$ belongs to $\Lambda (V_{<i}\otimes
B^\sharp)$, and
$$\Lambda (V_{<i}\otimes B^\sharp)=\Lambda \{ w_{ms}\}_{m<i}\otimes \Lambda\{ u_{mj}, \widetilde{d}u_{mj}\}_{m<i},$$
we can write $\widetilde{d}(v_i\otimes h_s)=\Gamma_1+\Gamma_2$, where
$$\Gamma_1\in \Lambda \{ w_{ms}\}_{m<i}\ \text{and}\ \Gamma_2\in \Lambda \{ w_{ms}\}_{m<i}\otimes \Lambda^+\{ u_{mj}, \widetilde{d}u_{mj}\}_{m<i}.$$
The ideal $\Lambda \{ w_{ms}\}_{m<i}\otimes \Lambda^+\{ u_{mj},
\widetilde{d}u_{mj}\}_{m<i}$ is acyclic, since by inductive hypothesis
$\Lambda \{ w_{ms}\}_{m<i}$ is $\widetilde{d}$-stable. Therefore, $\Gamma_2$
is a boundary, i.e., $\Gamma_2=\widetilde{d}x_{is}$ for some $x_{is}\in
\Lambda \{ w_{ms}\}_{m<i}\otimes \Lambda^+\{ u_{mj},
\widetilde{d}u_{mj}\}_{m<i}$. We define $w_{is}=v_i\otimes h_s-x_{is}$, which
clearly satisfies (ii), (iii) and (iv). To finish, observe that
$\widetilde{d}u_{ij}=\widetilde{d}(v_i\otimes a_j)=\pm v_i\otimes
b_j+\rho^{-1}[(dv_i)\otimes a_j]$, where $\rho^{-1}[(dv_i)\otimes a_j]\in
\Lambda (V_{<i}\otimes B^\sharp )$. Hence,
$$
\Lambda (V_{\leq i}\otimes B^\sharp)=\Lambda \{ w_{ms}\}_{m\leq i}\otimes \Lambda\{ u_{mj}, \widetilde{d}u_{mj}\}_{m\leq i},$$
and this completes the proof.
\end{proof}
We can endow the free algebra $\Lambda(V\otimes H)$ with a differential
$\widehat d$ so that the map
$$
\sigma \colon (\Lambda(V\otimes H),\widehat
d)\stackrel{\cong}{\longrightarrow}(\Lambda W,\widetilde d)$$ is
an isomorphism of differential graded algebras, and therefore
$(\Lambda(V\otimes H),\widehat d)$ is a model of $\map(X,Y)$ called the
\emph{reduced Brown--Szczarba model}. The differential $\widehat{d}$ of this
model can be easily described. The case  $(\Lambda V, d=d_1+d_2)$ is proved
in \cite[Lemma~2.8]{Bui09} and the same proof also works in the general case.

Consider the composition $\theta =\sigma^{-1}p\colon \Lambda (V\otimes
B^\sharp )=\Lambda W\otimes \Lambda (U\oplus \widetilde{d}U)\to \Lambda
(V\otimes H)$, where $p\colon \Lambda W\otimes \Lambda (U\oplus
\widetilde{d}U)\stackrel{\simeq }{\to }\Lambda W$ is the projection.
\begin{lemma}[\cite{Bui09}] The morphism $\theta $ operates on the generators as follows:
$$\theta (v\otimes b)=\left\{
\begin{array}{cl}
v\otimes b & \mbox{if $b\in H$}, \\
0& \mbox{if $b\in A$}, \\
(-1)^{|v|+1}\sum_k\sum_{i,j}\varepsilon \theta (v_i^{(1)}\otimes z_j^{(1)})\cdots \theta (v_i^{(k)}\otimes z_j^{(k)}) & \mbox{if $b\in \delta A$},
\end{array}
\right.$$
where in the last case $b=\delta a$, $\Delta^{(k-1)}a=\sum_j z_j^{(1)}\otimes
\cdots \otimes z_j^{(k)}$, $d_kv=\sum_i v_i^{(1)}\cdots v_i^{(k)}$, and
$\varepsilon$ is the sign given by Koszul convention. $\hfill\qed$
 \label{Theta_lemma}
\end{lemma}
\begin{lemma}[\cite{Bui09}]
The differential in the free commutative differential graded algebra
$(\Lambda (V\otimes H), \widehat{d})$ is given by the formula:
\begin{equation}\label{Hat differential}
\widehat{d}(v\otimes h)=d_1v\otimes h+\sum_k\sum_{i,j}\varepsilon \theta (v_i^{(1)}\otimes z_j^{(1)})\cdots \theta (v_i^{(k)}\otimes z_j^{(k)}),
\end{equation}
where $\Delta^{(k-1)}h=\sum_j z_j^{(1)}\otimes \cdots \otimes z_j^{(k)}$ and
$d_kv=\sum_i v_i^{(1)}\cdots v_i^{(k)},\ k\geq 2$. $\hfill\qed$ \label{Hat
differential lemma}
\end{lemma}

{We can describe a formula for $\widehat{d}_j$ in terms of directed rooted
trees. Let $T\in T_j$ be a directed rooted tree with $j$ leaves. We can
associate to $T$ two different linear maps $T_V\colon V\to V^{\otimes j}$ and
$T_H\colon H\to H^{\otimes j}$ moving up from the root to the leaves. In the
first case we label each internal vertex with $r$ input edges with
$\mathcal{S}\circ d_r$, where $\mathcal{S}\colon \Lambda^rV\to V^{\otimes r}$
is the symmetrization map. In the second case we label the root with $i$,
each internal vertex with $r$ input edges with $\Delta^{(r-1)}$, each
internal edge with $k$ and each leaf with $p$, where $i$, $p$ and $k$ are
described in the homotopy retract~$(\ref{homotopy_retract_01})$, taking
$C=B^\sharp$. Here is an example of the previous labeling defining $T_V$ and
$T_H$ respectively, for a concrete tree: }
$$\xymatrixcolsep{.5pc} \xymatrixrowsep{.5pc}
\entrymodifiers={=<1pc>} 
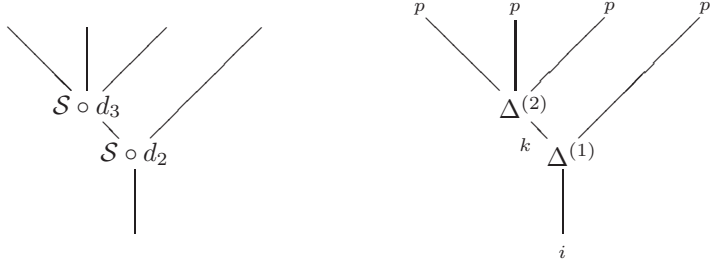
\xymatrix{ & \ar@{-}[ddrr]&
& \ar@{-}[dd]      &
&\ar@{-}[ddll]&&\ar@{-}[dddlll]&                           &
&*{^{p}}\ar@{-}[ddrr]&&*{^{p}}\ar@{-}[dd]
&&*{^{p}}\ar@{-}[ddll]&&*{^{p}}\ar@{-}[ddll]  \\
           &                & \ar@{-}[dr]&                       &                    &                  &&                    &          &
&&&&&&\ar@{-}[ddll]&\\
&                &                      & \mathcal{S}\circ d_3\ar@{-}[dr] &                    &                  &&                    &   & &&&\ \ \Delta^{(2)}\ar@{-}[dr]_{k}&&&&\\
           &                &                      &                       &\mathcal{S}\circ d_2\ar@{-}[dd]  &                  &&             &       &   &&&&\ \ \Delta^{(1)}\ar@{-}[dd]&&&\\
           &                &                      &                       & &           &          &&                   &   &&&& &&&
  \\ &&&&&&&&&&&&&{_{\stackrel{}{i}}}&&& \\
 }
$$
Now we define a linear map
\begin{equation}\label{Theta}
\Theta_j\colon V^{\otimes j}\times H^{\otimes j}\to \Lambda^j (V\otimes H),
 \end{equation}
 by $\Theta _j(v_1\otimes \cdots \otimes v_j, h_1\otimes \cdots \otimes h_j)=\varepsilon (v_1\otimes h_1)\cdots (v_j\otimes h_j)$, where $\varepsilon$ is the sign given by the Koszul convention.

Using the formula described in Lemma~\ref{Hat differential lemma}, we can
check the following:
\begin{lemma}\label{jth part of d}
The $j$th part of the differential in the free commutative differential
graded algebra $(\Lambda (V\otimes H), \widehat{d})$ is given by the formula:
\begin{equation}
\pushQED{\qed}
\widehat{d}_j(v\otimes h)=\sum_{T\in T_j}\frac{1}{|\Aut(T)|}\Theta_j (T_V(v), T_H(h)). \qedhere
\popQED
\end{equation}
\end{lemma}

If $C$ is a coalgebra model of $X$ and $L$ is an $L_{\infty}$-model of $Y$,
then we can endow the complex $\Hom(C,L)$ with an $L_{\infty}$-algebra
structure modeling  the space of continuous functions from $X$ to $Y$. More
concretely,
\begin{theorem}[\cite{Ber, BFM}] The complex of linear maps $\Hom(C,L)$ with brackets
$$\begin{array}{l} \ell_1(f)=\ell_1\circ f+(-1)^{|f|+1}f\circ \delta, \\
\ell_k(f_1,\dots ,f_k)=[-,\stackrel{(k)}{\ldots},-]_L\circ (f_1\otimes \cdots
\otimes f_k)\circ \Delta^{(k-1)},\ k\geq 2, \end{array}$$
is an $L_{\infty}$-algebra modeling $\map (X,Y)$. $\hfill\qed$
\label{thm:explicit_l}
\end{theorem}
Similarly, the complex $\homo (\overline{C},L)$, where $\overline{C}=\ker
\varepsilon$ is the kernel of the augmentation $\varepsilon \colon C\to
\mathbb{Q}$, with the same brackets replacing $\Delta $ with
$\overline{\Delta}$, is an $L_{\infty}$-model for $\map^*(X,Y)$, the space of
based functions.

\section{An explicit $L_\infty$-structure on ${\rm Hom}(H_*(X,\mathbb{Q}),\pi_*(\Omega Y)\otimes \mathbb{Q})$}
\label{main_section} In this section, we describe an explicit
$L_{\infty}$-structure on the complex of linear maps $\Hom(H, L)$, where $L$
is an $L_{\infty}$-algebra and $H$ is the homology of a coalgebra $C$ with
the transferred $A_{\infty}$-structure.

If $C$ is a cocommutative differential graded coalgebra and $H\cong H(C)$
denotes the homology of $C$, then the transferred $A_{\infty}$-coalgebra
structure on $H$, whose higher comultiplications are called \emph{higher
Massey coproducts} (cf.\ \cite[10.3.12]{LV}) is described as follows. We can
decompose $(C,\delta ,\Delta)$ as $A\oplus \delta A\oplus H$ with basis
$\{a_j\}$, $\{\delta a_j\}$ and $\{h_s \}$. Thus, $\delta=0$ in $H$ and
$\delta\colon A\to \delta A$ is an isomorphism. This decomposition induces a
homotopy retract
\begin{equation}
\xymatrix{ \ar@(ul,dl)@<-7ex>[]_k  & (C,\delta )
\ar@<0.75ex>[r]^-p & (H, 0) \ar@<0.75ex>[l]^-i }
\label{homotopy_retract_01}
\end{equation}
given by $p(a_j)=p(\delta a_j)=0$, $p(h_s)=h_s$; $i(h_s)=h_s$;
$k(a_j)=k(h_s)=0$ and $k(\delta a_j)=a_j$. Then by Theorem~\ref{htt}(i), we
can transfer the cocommutative differential graded coalgebra structure on $C$
to an $A_{\infty}$-coalgebra structure on $H$. For example, since $C$ has no
higher order coproducts, the operation $\Delta_3'$ on $H$ given by the
formula of Theorem~\ref{htt}(i) is provided by the trees
$$
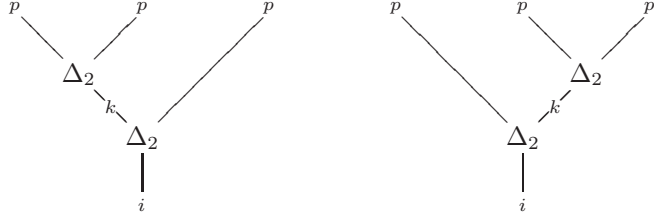
\xymatrixcolsep{1pc} \xymatrixrowsep{1pc} \entrymodifiers={=<1pc>}
\xymatrix{
 &*{^p}\ar@{-}[dr]&& *{^p}\ar@{-}[dl] & *{}&*{^p}\ar@{-}[ddll]& & *{^p}\ar@{-}[ddrr]&& *{^p}\ar@{-}[dr] & *{}&*{^p}\ar@{-}[dl]  \\
  && \Delta_2\ar@{-}[dr]|k&   &&&& && *{} & \Delta_2\ar@{-}[dl]|k &  \\
&  &*{} & \Delta_2\ar@{-}[d] & *{} && &&*{} & \Delta_2\ar@{-}[d] & *{} & \\
 &&*{} & *{_{\stackrel{}{i}}} & *{}& & &  &*{} & *{_{\stackrel{}{i}}} & *{} &
}
$$
Explicitly,
\begin{align*} \Delta_3'(h)&=(p\otimes p\otimes \text{id}
)\circ (\Delta_2\otimes \text{id} )\circ (k\otimes p)\circ\Delta_2
\circ i(h)\\
&\pm (\text{id}\otimes p\otimes p)\circ (\text{id}\otimes\Delta_2)
\circ (p\otimes k)\circ\Delta_2 \circ i(h)\\
&=\sum_j(p\otimes p\otimes \text{id} )\circ (\Delta_2\otimes
\text{id}
)\circ (k\otimes p)(z_j'\otimes z_j'')\\
&\pm \sum_j(\text{id}\otimes p\otimes p) \circ
(\text{id}\otimes\Delta_2)\circ (p\otimes k)(z_j'\otimes z_j''),
\end{align*}
where $\Delta_2(h)=\sum_j z_j'\otimes z_j''$. For a term of the form
$z'_j\otimes z''_j=\delta a\otimes h'$, the $j$th term in the above summation
is
\begin{align*}
& (p\otimes p\otimes \text{id} )\circ (\Delta_2\otimes \text{id}
)\circ (k(\delta a)\otimes p(h'))\\
&\pm (\text{id}\otimes p\otimes p)( \text{id}\otimes\Delta_2
)(p(\delta a)\otimes k(h'))\\
&=(p\otimes p\otimes \text{id} )\circ (\Delta_2\otimes \text{id}
)\circ (a\otimes h')\\
&=\sum_i(p\otimes p\otimes \text{id} )\circ (x_i'\otimes
x_i''\otimes h')\\
&=\sum_ip(x_i')\otimes p(x_i'')\otimes h',
\end{align*}
where $\Delta_2(a)=\sum_i x_i'\otimes x_i''$.

Replacing $C$ by the kernel of the augmentation $\overline{C}$ and using the
decomposition $\overline{C}=A\oplus \delta A\oplus\overline{H}$, we can
proceed similarly as above to obtain a transferred
$A_{\infty}$\nobreakdash-coalgebra structure on $\overline{H}$. The following
result relates the Quillen minimal model of a differential graded Lie algebra
with the above higher Massey coproducts:
\begin{theorem}
The transferred $A_{\infty}$-coalgebra structure on $\overline{H}$ is
cocommutative and $\mathcal{L}(\overline{H})$ is the Quillen minimal model of
$\mathcal{L}(\overline{C})$. \label{modelo_minimal}
\end{theorem}
\begin{proof}
Since the kernel of the augmentation $\overline{C}$ is cocommutative, the
transferred $A_{\infty}$-coalgebra structure on $\overline{H}$ is also
cocommutative (cf.\ \cite[Theorem 12]{ChG}). Then we can apply the Quillen
functor to diagram~(\ref{homotopy_retract_01}) and we get a
quasi-isomor\-phism
$(\mathbb{L}(s^{-1}\overline{H}),\partial)\stackrel{\sim}{\rightarrow}(\mathcal{L}(\overline{C}),\partial)$.
The $k$th part $\partial_k$ of the differential in
$(\mathbb{L}(s^{-1}\overline{H}),\partial)$ is determined by the higher
Massey coproduct $\overline{\Delta}'_k$ in $\overline{H}$.

Moreover, by the computations of $\overline{\Delta}'_k$ made above, if we
start by decomposing $(\overline{C},\delta,\overline{\Delta})$ as
$A\oplus\delta A\oplus \overline{H}$, where $\overline{H}\cong
H(\overline{C})$ with basis $\{a_j\}$, $\{\delta a_j\}$ and $\{h_s\}$, then
we can easily check that the differential on $\mathbb{L}(s^{-1}\overline{H})$
is described by
\begin{equation}
\partial s^{-1}h=\frac{1}{2}\sum_j(-1)^{|z_j'|}[\lambda(z_j'),
\lambda(z_j'')],
\label{req_for}
\end{equation}
where $\overline{\Delta}h=\sum_j z_j'\otimes z_j''$ and
$$
\lambda(h)= s^{-1}h,\quad \lambda(a)=0,\quad \lambda(\delta
a)=\frac{1}{2}\sum_i(-1)^{|x_i'|}[\lambda(x_i'),\lambda(x_i'')],
$$
where $\overline{\Delta}a=\sum_i x_i'\otimes x_i''$.

Finally, by~\cite[Theorem~2.1]{Bui09} or \cite[Theorem~22.13]{FHT},
$(\mathbb{L}(s^{-1}\overline{H}),\partial)$ agrees with the Quillen minimal
model of $(\mathcal{L}(\overline{C}),\partial)$.
\end{proof}
 The following is the main result of this article:
\begin{theorem}\label{main_thm}
Let $C$ be a finite dimensional cocommutative differential graded coalgebra
model of a finite nilpotent CW-complex $X$, and let $L$ be an
$L_{\infty}$-model of a rational CW-complex of finite type $Y$.
\begin{itemize}
\item[{\rm (i)}] There is an explicit $L_\infty$-structure in ${\rm
    Hom}(H,L)$, where $H$ denotes the homology of~$C$.
\item[{\rm (ii)}] With the above structure $\mathcal{C}^*(\Hom(H,L))\cong
    (\Lambda (V\otimes H), \widehat{d})$, the reduced Brown--Szczarba
    model of~$\map(X, Y)$, where $V=(sL)^{\sharp}$. Hence $\Hom(H, L)$ is
    an $L_\infty$-model for the mapping space.
\end{itemize}
If we replace $H$ by the homology $\overline{H}$ of the kernel of the
augmentation of~$C$, then we get the same results for the pointed mapping
space $\map^*(X,Y)$.
\end{theorem}
\begin{proof}
The homotopy retract~(\ref{homotopy_retract_01}) induces a new homotopy
retract
$$ \xymatrix{ \ar@(ul,dl)@<-7ex>[]_{k^*}  & (\text{Hom}(C,L), \delta )
\ar@<0.75ex>[r]^-{i^*} & (\text{Hom}(H,L), \delta ),
\ar@<0.75ex>[l]^-{p^*} }
$$
where both complexes of linear maps have the usual differentials. Since $i$
and $p$ are quasi-isomorphism so are $i^*$ and $p^*$. Observe that
$\text{Hom}(C,L)$ can be endowed with a natural $L_{\infty}$-algebra
structure modeling mapping spaces under suitable conditions~\cite{Ber, BFM}.
Hence we can apply Theorem~\ref{htt}(ii) to obtain an $L_{\infty}$-structure
$\{\ell_k'\}$ on $\Hom(H,L)$. This proves part (i).

In order to prove (ii), it is enough to see that equation~(\ref{pairing})
holds, which in this case amounts to check that
$$
\langle \widehat{d}_j(v\otimes h); sf_1\wedge\cdots\wedge sf_j\rangle=(-1)^{\epsilon}\langle
v\otimes h; s\ell'_j(f_1,\dots, f_j)\rangle.
$$
In what follows we work modulo signs and summations where necessary in order
to simplify the computations. In view of the isomorphism $V\otimes H\cong
(s\Hom(H,L))^{\sharp}$, and Theorem \ref{htt}(ii), the right hand part of the
above equation can be written as:
\begin{align*}
\langle v\otimes h; s\ell'_j(f_1,\dots, f_j)\rangle &= \pm \langle v ; s\ell'_j(f_1,\dots, f_j)(h)\rangle  \\
&= \pm \sum_{T\in T_k} \frac{1}{|\text{Aut}(T)|}\langle v ; s\ell_T (f_1,\dots , f_j)(h)\rangle \\
&=\sum _{T\in T_j}\langle \frac{1}{|\text{Aut}(T)|}\Theta_j (T_V(v),T_H(h)) ;sf_1\wedge \dots \wedge sf_j \rangle ,
\end{align*}
where we have used in the last equality equation~(\ref{pairing}). The result
is now a consequence of Lemma $\ref{jth part of d}$. The same proof works for
the pointed case.
\end{proof}
\begin{example}
In this example, we describe explicitly $\ell'_j$ in the case $j=3$. In what
follows we have omitted the suspensions in order to simplify the notation.
The explicit formula for $\ell_3'$ is provided by the trees
$$
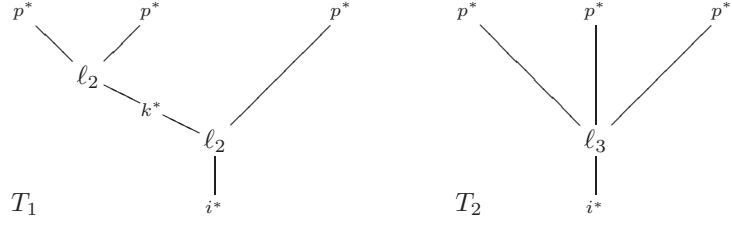
\xymatrixcolsep{1pc} \xymatrixrowsep{1pc} \entrymodifiers={=<1pc>}
\xymatrix{
 &*{^{p^*}}\ar@{-}[dr]& & *{^{p^*}}\ar@{-}[dl] & *{}&*{}& *{^{p^*}}\ar@{-}[ddll]&
 & *{^{p^*}}\ar@{-}[ddrr]&& *{^{p^*}}\ar@{-}[dd] & *{}&*{^{p^*}}\ar@{-}[ddll]  \\
 && \ell_2\ar@{-}[drr]|{k^*}&   && && & && *{} && \\
 &&*{} & *{} & \ell_2\ar@{-}[d] && & &&*{} & \ell_3\ar@{-}[d] & *{} & \\
 &T_1&*{} & *{} & *{_{\stackrel{}{i^*}}}& & & &T_2  &*{} & *{_{\stackrel{}{i^*}}} & *{} &
}
$$
Therefore, if $f_1,f_2$ and $f_3$ are elements of ${\rm Hom}(H,L)$ and $h$ is
an element of $H$, then $\ell_3'(f_1,f_2,f_3)(h)$ is expressed in terms of
the maps
\begin{align*}
\ell_{\widetilde{T}_2}(f_1,f_2,f_3)(h) & = i_*\ell_3(p^*f_1,p^*f_2,p^*f_3)(h) \\ &= [-,-,-]_L\circ (p^*f_1\otimes
p^*f_2\otimes p^*f_3)\circ \Delta^{(2)}(h)\\
&=\sum_j(-1)^{|z_j'|(|f_2|+|f_3|)+|z_j''||f_3|}[f_1p(z_j'),f_2p(z_j''),f_3p(z_j''')]_L,
\end{align*}
where $\Delta^{(2)}(h)=(\Delta \otimes \text{id})\circ \Delta
(h)=\sum_jz_j'\otimes z_j''\otimes z_j'''$, and
\begin{align*}
\ell_{\widetilde{T}_1}(f_1,f_2,f_3)(h)&=i^*\ell_2(k^*\ell_2(p^*f_1,p^*f_2),p^*f_3)(h)\\
&=[-,-]_L\circ (k^*\ell_2(p^*f_1,p^*f_2)\otimes p^*f_3)\circ \Delta
(h)\\
&=\sum_j(-1)^{|z_j'||f_3|}[\ell_2(p^*f_1,p^*f_2)\circ k(z_j'),f_3p(z''_j)]_L,
\end{align*}
where $\Delta (h)=\sum_jz_j'\otimes z_j''$. For a term of the form
$z_j'\otimes z_j''=\delta a\otimes h'$ the $j$th term in the above summation
equals
$$
\sum_{i}(-1)^{(|a|+1)|f_3|+|x_i'||f_2|}[[f_1p(x_i'),f_2p(x_i'')]_L,f_3 h']_L,
$$
where $\Delta(a)=\sum_ix_i'\otimes x_i''$.

We also make explicit the computations to check that
$$\langle v\otimes h; \ell'_3(f_1, f_2, f_3)\rangle=\pm \langle \widehat{d}_3(v\otimes h); f_1\wedge f_2\wedge f_3\rangle.$$
In order to simplify these computations, in what follows we set $\Delta
(h)=\delta a\otimes h'\pm h'\otimes  \delta a$ and $\Delta (a)=x'\otimes
x''\pm x''\otimes x'$.

First we compute $\widehat{d}_3(v\otimes h)$:
\begin{align*}
(T_1)_V(v)&=\Bigl((\mathcal{S}\circ d_2)\otimes \text{id}\Bigr)\circ (\mathcal{S}\circ d_2)(v)= \Bigl( (\mathcal{S}\circ d_2)\otimes \text{id}\Bigr)(u\otimes w\pm w\otimes u)\\
&=(u'\otimes u''\otimes w)\pm (u''\otimes u'\otimes w)\pm (w'\otimes w''\otimes u)\pm (w''\otimes w'\otimes u).\\
(T_2)_V(v)&=(\mathcal{S}\circ d_3)(v)\\
&=(p\otimes q\otimes r)\pm (p\otimes r\otimes q)\pm (q\otimes p\otimes r)\pm (q\otimes r\otimes p)\\ &\pm (r\otimes p\otimes q)\pm (r\otimes q\otimes p).\\
(T_1)_H(h)&=(p\otimes p\otimes \text{id})\circ (\Delta^{(1)}\circ k)\otimes p)(\delta a\otimes h'\pm h'\otimes \delta a)\\
&=(p\otimes p\otimes \text{id})(\Delta^{(1)}(a)\otimes h')=p(x')\otimes p(x'')\otimes h'\pm p(x'')\otimes p(x')\otimes h'.\\
(T_2)_H(h)&=(p\otimes p\otimes p)\circ \Delta^{(2)}\circ i (h)=\sum_jp(z_j')\otimes p(z_j'')\otimes p(z_j''').
\end{align*}
We can now apply Lemma \ref{jth part of d} or Lemma \ref{Hat differential
lemma} to obtain
\begin{align*}
\widehat{d}_3(v\otimes h)&=\frac{1}{|\text{Aut}(T_1)|}\Theta \Bigl( (T_1)_V(v), (T_1)_H(h)\Bigr)+\frac{1}{|\text{Aut}(T_2)|}\Theta \Bigl( (T_2)_V(v), (T_2)_H(h)\Bigr)\\
&=(u'\otimes p(x'))(u''\otimes p(x''))(w\otimes h')\pm (u'\otimes p(x''))(u''\otimes p(x'))(w\otimes h')\\
&\pm (u\otimes h')(w'\otimes p(x'))(w''\otimes p(x''))\pm (u\otimes h')(w'\otimes p(x''))(w''\otimes p(x'))\\
&\pm \sum_j(p\otimes p(z_j'))(q\otimes p(z_j''))(r\otimes p(z_j''')).
\end{align*}
Then, on the one hand we have
\begin{align*}
\langle \widehat{d}_3 (v\otimes h); f_1\wedge f_2\wedge f_3\rangle &=\sum_{\sigma \in S_3}\{\langle u';f_{\sigma (1)}(px')\rangle \langle u'';f_{\sigma (2)}(px'')\rangle \langle w; f_{\sigma (3)}(h')\rangle\\
&\pm \langle u';f_{\sigma (1)}p(x'')\rangle \langle u'';f_{\sigma (2)}p(x')\rangle \langle w; f_{\sigma (3)}(h')\rangle\\
&\pm \langle u;f_{\sigma (1)}(h')\rangle \langle w';f_{\sigma (2)}p(x')\rangle \langle w''; f_{\sigma (3)}p(x'')\rangle\\
&\pm \langle u;f_{\sigma (1)}(h')\rangle \langle w';f_{\sigma (2)}p(x'')\rangle \langle w''; f_{\sigma (3)}p(x')\rangle \}\\
&\pm \sum_{\sigma \in S_3}\sum_j\langle p; f_{\sigma (1)}p(z_j')\rangle \langle q; f_{\sigma (1)}p(z_j'')\rangle \langle r; f_{\sigma (3)}p(z_j''')\rangle\\
&=(\dagger)+(\ddagger).
\end{align*}
And on the other hand
$$
\langle v\otimes h; \ell'_3(f_1, f_2, f_3)\rangle=\pm\frac{1}{|\text{Aut}(T_1)|}\langle v;
\ell_{T_1}(f_1, f_2, f_3)(h)\rangle\pm \frac{1}{|\text{Aut}(T_2)|} \langle v; \ell_{T_2}(f_1, f_2, f_3)(h)\rangle.
$$
As we have seen before,
\begin{align*}
\ell_{\widetilde{T}_2}(f_1,f_2,f_3)(h) &= \pm\sum_j[f_1p(z_j'),f_2p(z_j''),f_3p(z_j''')]_L, \\
\ell_{\widetilde{T}_1}(f_1,f_2,f_3)(h) & =  \pm[[f_1p(x'),f_2p(x'')]_L,f_3(h')]_L\pm
[[f_1p(x''),f_2p(x')]_L,f_3(h')]_L,
\end{align*}
Then, applying~(\ref{pairing}) to the $L_{\infty}$-algebra $L$, we have
\begin{align*}
\frac{1}{|\text{Aut}(T_2)|}\langle v;
\ell_{T_2}(f_1, f_2, f_3)(h)\rangle&=\pm \frac{1}{3!}\sum_{\sigma \in S_3}\sum_j \langle v; [f_{\sigma (1)}p(z_j'), f_{\sigma (2)}p(z_j''), f_{\sigma (3)}p(z_j''')]_L\rangle \\
&=\pm \frac{1}{6}\sum_{\sigma \in S_3}\sum_j \langle pqr ; f_{\sigma (1)}p(z_j')\wedge f_{\sigma (2)}p(z_j'')\wedge  f_{\sigma (3)}p(z_j''')\rangle \\
&=(\ddagger);
\end{align*}
\begin{align*}
 &\frac{1}{|\text{Aut}(T_1)|}\langle v;
\ell_{T_1}(f_1, f_2, f_3)(h)\rangle\\
&=\pm \frac{1}{2}\sum_{\sigma \in S_3}\{ \langle v; [[f_{\sigma (1)}px', f_{\sigma (2)}px''], f_{\sigma (3)}h']\rangle \pm \langle v; [[f_{\sigma (1)}px'', f_{\sigma (2)}px'], f_{\sigma (3)}h']\rangle \}\\
&=\pm \frac{1}{2}\sum_{\sigma \in S_3}\{ \langle uw; [f_{\sigma (1)}px', f_{\sigma (2)}px'']\wedge f_{\sigma (3)}h'\rangle \pm \langle uw; [f_{\sigma (1)}px'', f_{\sigma (2)}px)]\wedge f_{\sigma (3)}h'\rangle \}\\
&=\pm \frac{1}{2}\sum_{\sigma \in S_3}\{\langle u;[f_{\sigma (1)}px', f_{\sigma (2)}px'']\rangle \langle w; f_{\sigma (3)}ph'\rangle\pm \langle w;[f_{\sigma (1)}px', f_{\sigma (2)}px'']\rangle \langle u; f_{\sigma (3)}ph'\rangle\\
&\pm \langle u;[f_{\sigma (1)}px'', f_{\sigma (2)}px']\rangle \langle w; f_{\sigma (3)}ph'\rangle \pm \langle w;[f_{\sigma (1)}px'', f_{\sigma (2)}px']\rangle \langle u; f_{\sigma (3)}ph'\rangle \}\\
&=\pm \frac{1}{2}\sum_{\sigma \in S_3}\{\langle u'u'';f_{\sigma (1)}px'\wedge f_{\sigma (2)}px''\rangle \langle w; f_{\sigma (3)}ph'\rangle \\
&\pm \langle w'w'';f_{\sigma (1)}px'\wedge f_{\sigma (2)}px''\rangle \langle u; f_{\sigma (3)}ph'\rangle\\
&\pm \langle u'u'';f_{\sigma (1)}px'', f_{\sigma (2)}px'\rangle \langle w; f_{\sigma (3)}ph'\rangle \pm \langle w'w'';f_{\sigma (1)}px'', f_{\sigma (2)}px'\rangle \langle u; f_{\sigma (3)}ph'\rangle \}=(\dagger).\\
\end{align*}
\end{example}
\section{Examples and applications}

We will work with models of the components of the mapping space, and for a
based map $f\colon X\to Y$, we will denote by $\map^*_f(X,Y)$ the component
containing~$f$. These $L_{\infty}$-models can be obtained via the process of
perturbation and truncation described in Section~\ref{sect:prelim}. More
explicitly, let $\varphi\colon \mathcal{L}(\overline{C})\to L$ be an
$L_{\infty}$-model of $f\colon X\to Y$. Then, the composite
$$
\psi\colon\mathcal{L}(\overline{H})\longrightarrow
\mathcal{L}(\overline{C})\longrightarrow L
$$
is also a model for~$f$. Hence, as explained in~\cite{BFM}, the induced
degree $-1$ linear map $\overline{C}\to L$, which we will keep denoting by
$\varphi$, is a Maurer--Cartan element of the
$L_{\infty}$\nobreakdash-algebra $\Hom(\overline{C},L)$, and the induced map
$\overline{H}\to \overline{C}\to L$ is a Maurer--Cartan element in
$\Hom(\overline{H},L)$. Therefore, the perturbed and truncated
$L_{\infty}$-algebras $\Hom(\overline{C},L)^{(\varphi)}$ and
$\Hom(\overline{H},L)^{(\psi)}$ are $L_{\infty}$-models for $\map_f^*(X,Y)$.

\subsection{Formality, coformality and mapping spaces}
We recall that $X$ is {\em formal} if $H^*(X;\mathbb{Q})$ equipped with the
cup product is a differential graded algebra model for the space $X$ or,
equivalently, if $X$ has  a Quillen minimal model $(\mathbb{L}(V),\partial )$
whose differential is quadratic $\partial=\partial_2$. In the same way $Y$ is
{\em coformal} if $\pi_*(\Omega Y)\otimes \mathbb{Q}$ equipped with the
Samelson product is a differential graded Lie algebra  model for $Y$ or,
equivalently, if $Y$ has a Sullivan minimal model $(\Lambda V, d)$ whose
differential is quadratic $d=d_2$.

The graded vector space which carries the $L_\infty$-structure modeling the
mapping space is of the form
$$
\text{Hom}(H_*(X;\mathbb{Q}),\pi_*(\Omega Y)\otimes \mathbb{Q}),
$$
although the source and target spaces $X$ and $Y$ are not necessarily formal
and coformal. We will now show examples in which the source and/or the target
are not formal and coformal.

\begin{example}
Consider the rational space $X$ with Sullivan model $A=(\Lambda (a, b, c),
d)$, where $|a|=|b|=3$, $|c|=5$, $da=db=0$ and $dc=ab$. Note that this space
is coformal but not formal as we will see below. Let $Y$ be the rational
space with Sullivan model $(\Lambda (x, y, z, t), d)$ where $|x|=4$, $|y|=7$,
$|z|=10$, $|t|=16$ and $dx=dy=0$, $dz=xy$, $dt=yz$.

We will describe the rational homotopy type of $\text{map}^*(X, Y)$ by
studying the $L_\infty$\nobreakdash-struc\-tu\-re on
$\text{Hom}(H_+(X;\mathbb{Q}),\pi_*(\Omega Y)\otimes \mathbb{Q})$. Note that
this mapping space has only one component, the one which contains the
constant map since the only morphism of differential graded algebras
$$\varphi \colon (\Lambda (x, y, z, t), d)\to (\Lambda (a, b, c), d)$$
is the zero morphism $\varphi =0$ by degree reasons.

We have that, as graded Lie algebras,
$$\pi_*(\Omega \text{map}^*(X, Y))\otimes \mathbb{Q}\cong \text{Hom}(H_+(X; \mathbb{Q}), \pi_*(\Omega Y)\otimes \mathbb{Q}),$$
since the fact that $\varphi=0$ implies that the bracket $\ell_1$ in the
$L_\infty$-structure given by Theorem \ref{main_thm} is zero, or
equivalently, the differential in the reduced Brown-Szczarba model has no
linear term. However, as we will show the mapping space is not coformal since
the differential in its Sullivan model is not purely quadratic.

In particular, we will see that $\text{map}^*(X, Y)$ splits rationally as
$$
Z\times (S^7\times S^7\times S^{13}\times S^{13}),
$$
where $Z$ is the rational space with minimal model
$$
(\Lambda (a_1, b_1, a_2, b_2, a_4, b_4, z_5, x_8, y_8), d),
$$
where $da_1=\cdots =db_4=0$, $dz_5=a_4b_2-b_4a_2$,
$dx_8=-a_4a_1b_4+a_4^2b_1$, and $dy_8=-b_4^2a_1+b_4b_1a_4$.

This calculation can be made directly from the Haefliger \cite{Hae82} or the
Brown-Szczarba model \cite{BS97} but, in order to illustrate the homotopy
transfer techniques described in this paper, we will make use of Theorem
\ref{main_thm} instead.

First we describe the $A_\infty$-coalgebra structure induced by the higher
Massey coproducts on $H_+(X;\mathbb{Q})$. As a graded vector space,
$\overline{A}$ is given by
$$\overline{A}=A^3\oplus A^5\oplus A^6\oplus A^8\oplus A^{11}=\langle a, b\rangle
\oplus \langle c \rangle \oplus \langle ab\rangle \oplus \langle ac, bc\rangle \oplus \langle abc\rangle.$$
The dual coalgebra $\overline{C}=(\overline{A})^\sharp$ is
$$\overline{C}=C_3\oplus C_5\oplus C_6\oplus C_8\oplus C_{11}=\langle g, h\rangle
\oplus \langle r \rangle \oplus \langle s\rangle \oplus \langle u, v\rangle \oplus \langle w\rangle,$$
with $\delta s=r$ and reduced comultiplications
\begin{align*}
\overline{\Delta }g=\overline{\Delta }h=\overline{\Delta }r&=0,\\
\overline{\Delta }s&=g\otimes h -h \otimes g,\\
\overline{\Delta }u&=g\otimes r-r\otimes g,\\
\overline{\Delta }v&=h \otimes r-r\otimes h, \\
\overline{\Delta }w&=g\otimes v-v\otimes g-h\otimes u+u\otimes h +s\otimes r+r\otimes s.
\end{align*}
The decomposition $\overline{C}=S\oplus \delta S\oplus \overline{H}$, where
$S=\langle s\rangle$, $\delta S=\langle r\rangle$ and $\overline{H}=\langle
g, h\rangle \oplus  \langle u, v\rangle \oplus \langle w\rangle$ induces the
homotopy retract
\begin{equation*}
\xymatrix{ \ar@(ul,dl)@<-7ex>[]_k  & (\overline{C},\delta )
\ar@<0.75ex>[r]^-p & (\overline{H}, 0), \ar@<0.75ex>[l]^-i }
\end{equation*}
where $k(r)=s$. Note that $X$ is not formal since the transferred structure
in $\overline{H}$ (which gives rise to the differential on the Quillen
minimal model as  shown in Theorem \ref{modelo_minimal}) has non-zero higher
order coproducts. Indeed,
\begin{align*} \overline{\Delta}_3'(u)&=(p\otimes p\otimes \text{id}
)\circ (\overline{\Delta} \otimes \text{id} )\circ (k\otimes p)\circ\overline{\Delta}
\circ i(u)\\
&+(\text{id}\otimes p\otimes p)\circ (\text{id}\otimes\overline{\Delta})
\circ (p\otimes k)\circ\overline{\Delta} \circ i(u)\\
&=-(p\otimes p\otimes \text{id}) ( \overline{\Delta }(s)\otimes g)+(\text{id}\otimes p\otimes p)(g\otimes \overline{\Delta }(s))\\
&=(g\otimes g\otimes h)-2(g\otimes h\otimes g)+(h\otimes g\otimes g).
\end{align*}

On the other hand, the model for $Y$ is of the form
$\mathcal{C}^*(L)=(\Lambda (x, y, z, t), d)$ for some $L_\infty$-algebra $L$
with $\langle x, y, z, t\rangle \cong (sL)^\sharp$, and brackets induced by
the differential $d$ by formula \eqref{pairing}. We write $L=\langle x', y',
z', t'\rangle $ with $|x'|=3$, $|y'|=6$, $|z'|=9$, $|t'|=15$.

With the above data we can define a homotopy retract
\begin{equation*}
\xymatrix{ \ar@(ul,dl)@<-7ex>[]_{k^*}  & \text{Hom}(\overline{C},\langle x', y', z', t'\rangle )
\ar@<0.75ex>[r]^-{i^*} & \text{Hom}(\overline{H}, \langle x', y', z', t'\rangle) \ar@<0.75ex>[l]^-{p^*}
}
\end{equation*}
where $\text{Hom}(\overline{H}, \langle x', y', z', t'\rangle)\cong
\text{Hom}(H_+(X;\mathbb{Q}),\pi_*(\Omega Y)\otimes \mathbb{Q})$ as graded
vector spaces. Then we have that $ \text{Hom}(H_+(X;\mathbb{Q}),\pi_*(\Omega
Y)\otimes \mathbb{Q})=K_*$, with
$$
K_*=K_{-8}\oplus K_{-5}\oplus K_{-2}\oplus K_0\oplus K_1\oplus K_3\oplus K_4\oplus K_6\oplus K_7\oplus K_{12},
$$
where $K_{-8}=\langle f_{w}^{x'}\rangle$ with $f_{w}^{x'}(w)=x'$ and
$f_{w}^{x'}(\Phi )=0$ if $\Phi $ is any other basis element, and
\begin{align*}
K_{-5}=\langle f_{w}^{y'}, f_{u}^{x'}, f_{v}^{x'}\rangle ,\ K_{-2}=
\langle f_{u}^{y'}, f_{v}^{y'}, f_{w}^{z}\rangle,\  K_{0}=\langle f_{g}^{x'}, f_{h}^{x'} \rangle ,\ K_{1}=
\langle f_{u}^{z'}, f_{v}^{z'} \rangle,\\
K_{3}=\langle f_{g}^{y'}, f_{h}^{y'} \rangle ,\
K_{4}=\langle f_{w}^{t'}\rangle ,\
K_6=\langle f_{g}^{z'}, f_{h}^{z'} \rangle ,\
K_7=\langle f_{u}^{t'}, f_{v}^{t'} \rangle ,\
K_{12}=\langle f_{g}^{t'}, f_{h}^{t'} \rangle.
\end{align*}
The higher brackets can be computed using the explicit formula given in
Theorem~\ref{htt}(ii), but the only Maurer--Cartan element will be the zero
element since there is only one component. So, in order to compute the model
of the mapping space we must only discard the elements of negative degree. We
give some explicit computations. In order to compute $\ell'_2$ we have to
consider the tree
$$\xymatrixcolsep{1pc} \xymatrixrowsep{1pc} \entrymodifiers={=<1pc>}
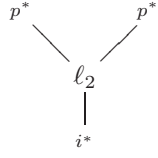
\xymatrix{
*{^{p^*}}\ar@{-}[rd]&&*{^{p^*}}\ar@{-}[ld]\\
&\ell_2 \ar@{-}[d]&\\
&{_{i^*}}&
}$$
For the element $w$ we have that
\begin{align*}
\ell'_2(f_g^{y'}, f_v^{z'})(w) &=(\ell_2)_L\circ \bigl( (f_g^{y'}\circ p)\otimes (f_v^{z'}\circ p)\bigr)\circ \overline{\Delta }(w)\\
 &=(\ell_2)_L\circ \bigl( (f_g^{y'}\circ p)\otimes (f_v^{z'}\circ p)\bigr)(g\otimes v-v\otimes g-h\otimes u\\
 &+u\otimes h+s\otimes r+r\otimes s)\\
 &=(\ell_2)_L(y',z')=t',
\end{align*}
and it is easy to check that $\ell'_2(f_g^{y'}, f_v^{z'})$ vanishes in the
rest of basis elements. Hence $\ell'_2(f_g^{y'}, f_v^{z'})=f_w^{t'}$. For
computing $\ell_3'$ we have to use the tree

$$
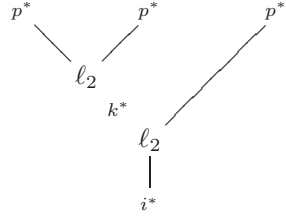
\xymatrixcolsep{1pc} \xymatrixrowsep{1pc} \entrymodifiers={=<1pc>}
\xymatrix{
*{^{p^*}}\ar@{-}[dr]&& *{^{p^*}}\ar@{-}[dl] & *{}&*{^{p^*}}\ar@{-}[ddll] \\
& \ell_2\ar@{-}[dr]|{k^*}&   && \\
  &*{} & \ell_2\ar@{-}[d] & *{} & \\
&*{} & {_{i^*}} & *{}&  }
$$
For the element $u$, we have that

\begin{align*}
\ell'_3(f_g^{y'}, f_h^{x'}, f_h^{y'})(u) &=i^*\ell_2 \bigl( k^*(\ell_2(p^*f_g^{y'}, p^*f_h^{x'})), p^*f_h^{y'}\bigr) (u)\\
&=(\ell_2)_L\circ (k^*(\ell_2(p^*f_g^{y'}, p^*f_h^{x'}))\otimes p^*f_h^{y'})\circ \overline{\Delta }(u)\\
&=(\ell_2)_L\circ (k^*(\ell_2(p^*f_g^{y'}, p^*f_h^{x'}))\otimes p^*f_h^{y'})(g\otimes r-r\otimes g)\\
&=(\ell_2)_L\bigl( \ell_2(p^*f_g^{y'}, p^*f_h^{x'})k(r), f_g^{y'}(g)\bigl)\\
&=(\ell_2)_L\bigl( \ell_2(p^*f_g^{y'}, p^*f_h^{x'})(s), f_g^{y'}(g)\bigl)\\
&=-(\ell_2)_L\bigl( (\ell_2)_L(f_g^{y'}(g), f_h^{x'}(h), y'\bigr)\\
&=-(\ell_2)_L\bigl( (\ell_2)_L(y', x'), y'\bigr)=(\ell_2)_L(z', y')=-t'.
\end{align*}
Again, it is easy to check that $\ell'_3(f_g^{y'}, f_h^{x'}, f_h^{y'})$
vanishes in the rest of basis elements. Hence $\ell'_3(f_g^{y'}, f_h^{x'},
f_h^{y'})=-f_u^{t'}$.

We can compute the cochain functor of this $L_\infty$-algebra
$$
\mathcal{C}^*(\text{Hom}(\overline{H},L) )=(\Lambda ((sL)^\sharp\otimes \overline{H}), d),
$$
 and the rational homotopy type of $\text{map}^*(X, Y)$ is modeled by
$$
(\Lambda ((sL)^\sharp\otimes \overline{H})^{>0}, \widehat{d})
$$
as described in Section~\ref{sect:rational_models_map}, which coincides with
the reduced Brown--Szczarba model. Explicitly, $((sL)^\sharp\otimes
\overline{H})^{>0}$ is concentrated in degrees $1,2,4,5,7,8$ and $13$,
 $$((sL)^\sharp\otimes \overline{H})^{1}=\langle x\otimes g, x\otimes h\rangle,\ ((sL)^\sharp\otimes \overline{H})^{2}=\langle z\otimes u, z\otimes v\rangle,$$
  $$((sL)^\sharp\otimes \overline{H})^{4}=\langle y\otimes g, y\otimes h\rangle,\ ((sL)^\sharp\otimes \overline{H})^{5}=\langle t\otimes w\rangle,$$
   $$((sL)^\sharp\otimes \overline{H})^{7}=\langle z\otimes g, z\otimes h\rangle,\ ((sL)^\sharp\otimes \overline{H})^{8}=\langle t\otimes u, t\otimes v\rangle,$$
   $$((sL)^\sharp\otimes \overline{H})^{13}=\langle t\otimes g ,t\otimes h \rangle,$$
and the differentials are given by
\begin{align*}
\widehat{d}(x\otimes g)&=\widehat{d}(x\otimes h)=\widehat{d}(z\otimes u)=\widehat{d}(z\otimes v)
=\widehat{d}(y\otimes g)=\widehat{d}(y\otimes h)\\
&=\widehat{d}(z\otimes g)=\widehat{d}(z\otimes h)=\widehat{d}(t\otimes g)=\widehat{d}(t\otimes h)=0,\\
\widehat{d}(t\otimes w)&=(y\otimes g)(z\otimes v)-(y\otimes h)(z\otimes u),\\
\widehat{d}(t\otimes u)&=-(y\otimes g)(x\otimes g)(y\otimes h)+ (y\otimes g)^2(x\otimes h),\\
\widehat{d}(t\otimes v)&=-(y\otimes h)^2(x\otimes g)+ (y\otimes h)(x\otimes h)(y\otimes g).
\end{align*}
Thus, we obtain the rational equivalence
$$\text{map}^*(X, Y)\simeq_\mathbb{Q} Z\times (S^7\times S^7\times S^{13}\times S^{13}).$$
\end{example}

\subsection{$H$-space structures} Another interesting question is to determine
whether a mapping space is of the rational homotopy type of an $H$-space,
that is, it has a Sullivan minimal model with zero differential, in terms of
the source and target spaces.

Recall that for a space $X$, the \emph{differential length} $\dl(X)$ is the
least integer $n$ such that there is a non-trivial Whitehead product of order
$n$ on $\pi_*(X)\otimes\mathbb{Q}$. This number coincides with the least $n$
for which the $n$-th part of the differential of the Sullivan minimal model
of $X$ is non trivial; see~\cite{AA}. If there is not such an $n$, then
$\dl(X)=\infty$. Dually, the \emph{bracket length} $\bl(X)$ is the length of
the shortest non-zero iterated bracket in the differential of the Quillen
minimal model of $X$. If the differential is zero, then $\bl(X)=\infty$.

The \emph{rational cone length} $\cl(X)$ is the least integer $n$ such that
$X$ has the rational homotopy type of an $n$-cone; see~\cite[p.\ 359]{FHT}.
The \emph{rational Whitehead length} $\Wl(X)$ is the length of the longest
non-zero iterated Whitehead bracket in $\pi_{\ge 2}(X)\otimes\mathbb{Q}$. In
particular, if $\Wl(X)=1$, then all Whitehead products vanish.

In~\cite[Theorem 2]{FT05} a necessary condition, in terms of the Toomer
invariant, is given in order to ensure that the component of the constant map
is an $H$-space. In~\cite[Theorem 4]{BM08} it was proved that if
$\cl(X)<\dl(Y)$ then \emph{all} the components $\map_f^*(X, Y)$ are
rationally $H$-spaces. A dual result (in the sense of Eckmann-Hilton) was
given in~\cite[Theorem~1.4(2)]{Bui09}, by assuming that $Y$ is a coformal
space. In this case, if $\Wl(Y)<\bl(X)$ then all the components $\map_f^*(X,
Y)$ are again rationally $H$-spaces. In this paper we formulate a variant of
this last result that does not implicitly assume the coformality of~$Y$.

\begin{theorem}\label{main02}
If $\cl(X)=2$ and $\Wl(Y)<\bl(X)$, then all the components of the mapping
space $\map^*(X,Y)$ are rationally $H$-spaces.
\end{theorem}

\begin{proof}
Let $L$ be the minimal $L_\infty$-model of $Y$. Following~\cite{Cor94} or
\cite[Theorem 29.1]{FHT}, and since $\cl(X)=2$, we may choose $C$ a coalgebra
model of $X$ with conilpotence~$2$, that is, such that the iterated
coproducts of length greater or equal than $2$ are zero. Consider the
$L_\infty$\nobreakdash-model $\Hom(H,L)^{(\phi )}$ of the component
corresponding to $\phi$ given in Theorem~\ref{main_thm}. In order to prove
that $\map_f^*(X,Y)$ is an $H$-space we will show that all brackets
$\ell_k^{\phi}$ for $k\geq 2$ vanish, and thus
$\mathcal{C}^{*}(\Hom(H,L)^{(\phi)})$ is of the form $(\Lambda S, d_1)$, for
which $(\Lambda H(S, d_1),0)$ is a Sullivan minimal model. It is enough to
check that $\ell'_k=0$ for all $k$ since this implies that
${\ell'}^{\phi}_k=0$ for all $k$.

Now, all the brackets are represented by summations on trees. But due to the
conilpotence condition on $C$ only the binary ones contribute to the
brackets. Then as we have pointed in Remark~\ref{modelo_minimal}, the
operations given by these trees are defined in terms of the differential in
the minimal model of $X$. Since the bracket length is the length of the
shortest non-zero iterated bracket appearing in the differential of the
Quillen minimal model, the corresponding binary tree is expressed by an
iterated bracket of this length in $L$, and hence it must vanish.
\end{proof}

We illustrate the last part of the proof by showing that $\ell_4'=0$. This
operation is defined as a sum of maps whose terms are indexed by the
following five trees:

{\tiny$$ \xymatrixcolsep{.3pc} \xymatrixrowsep{.3pc}
\entrymodifiers={=<.5pc>} 
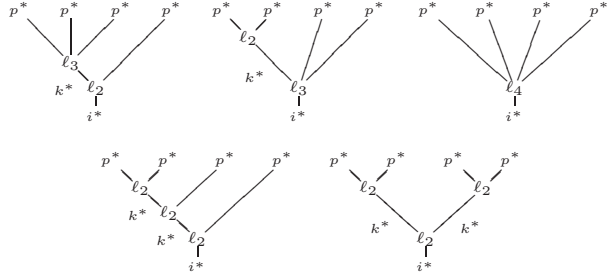
\xymatrix{
&{^{p^*}}\ar@{-}[ddrr]&&{^{p^*}}\ar@{-}[dd] & &{^{p^*}}\ar@{-}[ddll]&&{^{p^*}}\ar@{-}[dddlll]&   &{^{p^*}}\ar@{-}[dr]&&{^{p^*}}\ar@{-}[dl]&&{^{p^*}}\ar@{-}[dddl]   &                   &{^{p^*}}\ar@{-}[dddlll]  &   &*{^{p^*}}\ar@{-}[dddrrr]&&*{^{p^*}}\ar@{-}[dddr]  &                   &*{^{p^*}}\ar@{-}[dddl]     &                      &*{^{p^*}}\ar@{-}[dddlll]  \\
& && && &&  &  & &\ell_2\ar@{-}[ddrr]_{k^*}&  && & & & & && & & &&              \\
&  &&\ell_3\ar@{-}[dr]_{k^*}& & && & & && & &&& & & && & && &         \\
& && &\ell_2\ar@{-}[d] &               &&                   &   &                   &&               &\ell_3\ar@{-}[d] &&                   &                 &   &                    &&                    &\ell_4\ar@{-}[d] &                     &                      &                  \\
&                 &&                   &{_{\stackrel{}{i^*}}}& &&
&   &                   && &{_{\stackrel{}{i^*}}}&&                   &
&   & &&                    &{_{\stackrel{}{i^*}}}&                     & & }
$$}
\vspace{0.2cm}
{\tiny $$\xymatrixcolsep{.3pc} \xymatrixrowsep{.3pc}
\entrymodifiers={=<.5pc>} \xymatrix{ & *{^{p^*}}\ar@{-}[dr]& &
*{^{p^*}}\ar@{-}[dl]      &
&*{^{p^*}}\ar@{-}[ddll]&&*{^{p^*}}\ar@{-}[dddlll]&
&*{^{p^*}}\ar@{-}[dr]&&*{^{p^*}}\ar@{-}[dl]
&&*{^{p^*}}\ar@{-}[dr]&&*{^{p^*}}\ar@{-}[dl]  \\
           &                & \ell_2\ar@{-}[dr]_{k^*}&                       &                    &                  &&                    &
&&\ell_2\ar@{-}[ddrr]_{k^*}&&&&\ell_2\ar@{-}[ddll]^{k^*}&\\
&                &                      & \ell_2\ar@{-}[dr]_{k^*} &                    &                  &&                    &  &&&&&&&\\
           &                &                      &                       &\ell_2\ar@{-}[d]  &                  &&                    &   &&&&\ell_2\ar@{-}[d]&&&\\
           &                &                      &                       & {_{\stackrel{}{i^*}}}&                  &&                   &   &&&& {_{\stackrel{}{i^*}}}&&&
   }
$$}
Note that the first three trees do not contribute to $\ell'_4$ since they are not binary. For example, for the second tree we have that
\begin{gather*}
i^*\ell_3(k^*\ell_2(p^*f_1, p^* f_2), p^* f_3, p^* f_4)(h)  \\
=[-,-,-]_L\circ (k^*\ell_2(p^*f_1, p^* f_2)\otimes p^* f_3\otimes p^* f_4)\circ \overline{\Delta}^{(2)}(h)
\end{gather*}
is zero. Indeed, $\overline{\Delta}^{(2)}=(\Delta\otimes \text{id})\circ \Delta \colon C\to C\otimes C\otimes C$ is zero since $\cl(X)=2$.

The last two trees do not contribute to $\ell_4'$ either, since
$\Wl(Y)<\bl(X)$. Thus, for example, for any $h\in\overline{H}$, the fourth
tree yields
$$
i^*\ell_2(k^*\ell_2(k^*\ell_2(p^*f_1,p^* f_2), p^*f_3), p^* f_4)(h).
$$
A generic term in this expression is given by
\begin{equation}
\sum_{i}[[f_1 px'_i, f_2 px''_i]_L, f_3 py''_j]_L, f_4 pz''_l]_L,
\label{formulo}
\end{equation}
where $\overline{\Delta}(h)=\sum_l z'_l\otimes z_l''$, $z'_l=\delta a$ and
hence $kz'_l=a$, $\overline{\Delta}(a)=\sum_j y'_j\otimes y''_j$,
$y'_j=\delta a'$ and hence $ky'_j=a'$, and $\overline{\Delta}(a')=\sum_i
x'_i\otimes x''_i$. If $px'_i$, $px''_i$, $py''_j$, $pz''_l\ne 0$, then by
the recursive formula~(\ref{req_for}) of Proposition~\ref{modelo_minimal},
the term
$$
[[s^{-1}px'_i,s^{-1}px''_i], s^{-1}py'_j],s^{-1}pz''_l]
$$
appears in the expression of $\partial s^{-1}h$, where $\partial$ is the
differential of the Quillen minimal model of $X$. Therefore, (\ref{formulo})
must be zero, since $\Wl(Y)<\bl(X)$.

\begin{remark}
We can give an alternative proof of Theorem~\ref{main02} by relying on
\cite[Theorem~1.4(2)]{Bui09}, where the same result is obtained, but under
the assumptions $\Wl(Y)<\bl(X)$ and coformality of the target space $Y$.
Indeed, let $L$ be a minimal $L_{\infty}$-model of $Y$ and denote by $L^{\rm
cof}$ the differential graded Lie algebra obtained from~$L$ by discarding all
higher brackets except the binary one. The graded Lie algebra $L^{\rm cof}$
represents a coformal rational space $Y^{\rm cof}$ such that
$\Wl(Y)=\Wl(Y^{\rm cof})$. Let $C$ be a coalgebra model of $X$ such that
$\overline{C}$ has conilpotence $\cl(X)=2$. Then the convolution
$L_{\infty}$-algebra $\Hom(\overline{C}, L)$ is isomorphic to the convolution
differential graded Lie algebra $\Hom(\overline{C}, L^{\rm cof})$. Therefore
$\map^*(X,Y)\simeq \map^*(X, Y^{\rm cof})$. \label{ref_rem}
\end{remark}

\begin{example}
Let $X=S_a^7\vee S_b^7\cup_{\gamma} e^{20}$, let $\alpha$, $\beta\in\pi_7(X)$
be the elements represented by $S_a^7$ and $S_b^7$, respectively, and let
$\gamma=[\alpha,[\alpha,\beta]]\in\pi_{19}(X)$. Then, the Quillen minimal
model for $X$ is $(\mathbb{L}(W),\partial)$, where $W=\langle a,b,c\rangle$,
$|a|=|b|=6$ and the only non-zero differential is $\partial c=[a,[a,b]]$.
Hence, it is clear that $\bl(X)=3$. Moreover, $\cl(X)=2$ since $W=W_0\oplus
W_1$ with $\partial W_0=0$ and $\partial W_1\subseteq \mathbb{L}(W_0)$; see,
for example, \cite[Theorem 29.1]{FHT}.

Let $Y$ be the space with Sullivan minimal model of the form  $(\Lambda (u,
v, w), d)$, $|u|=2, |v|=4$, $|w|=7$, $du=dv=0$, $dw=u^4+v^2$.

Therefore, $Y$ is a non-coformal space with $\dl(Y)=\Wl(Y)=2$. Moreover,
there are non null-homotopic maps from $X$ to $Y$. For example, $\phi\colon
\Lambda V\to \mathcal{C}^{*}(\mathbb{L}(a,b,c))=\Lambda
s\mathbb{L}(a,b,c)^{\sharp}$, defined by $\phi(u)=\phi(v)=0$ and
$\phi(w)=sa^{\sharp}$.

Then, by Theorem~\ref{main02}, \emph{every} component of the mapping space
$\map^*(X,Y)$ is rationally an $H$-space. This example provides a situation
where we cannot apply neither~\cite[Theorem 4]{BM08}, since $\cl(X)=\dl(Y)$,
nor \cite[Theorem 1.4(2)]{Bui09}, since $Y$ is not coformal.
Also~\cite[Theorem 2]{FT05} does not provide any information for the
component of a non-constant map.
\end{example}


\begin{thebibliography}{99}

\bibitem{AA} P. Andrews and M. Arkowitz, Sullivan's minimal models and higher
    order Whitehead products, \textit{Canad. J. Math.} {\bf 30} (1978), no.
    5, 961--982.0

\bibitem{BR10} A. Berciano and P. Real, $A_{\infty}$-coalgebra structure maps
    that vanish on $H_*(K(\pi,n), \mathbb{Z}_p)$,  \textit{Forum Math.} {\bf
    22}, no. 2, (2010), 357--378.

\bibitem{Ber} A. Berglund, Rational homotopy theory of mapping spaces via Lie
    theory for $L_{\infty}$ algebras, {\it preprint} {\tt arXiv:1110.6145}.

\bibitem{BS97} E.\,H. Brown and R.\,H. Szczarba, On the rational homotopy
    type of function spaces, \textit{Trans. Amer. Math. Soc.} {\bf 349}
    (1997), 4931--4951.

\bibitem{Bui09} U. Buijs, An explicit $L_\infty$ structure for the components
    of mapping spaces, \textit{Topology Appl.} {\bf 159}, no. 3 (2012),
    721--732.

\bibitem{Bui10} U. Buijs, Upper bounds for the Whitehead length of mapping
    spaces. Homotopy theory of function spaces and related topics, 43--53,
    \textit{Contemp. Math.} 519, Amer. Math. Soc., Providence, RI, 2010.

\bibitem{BM08} U. Buijs and A. Murillo, The rational homotopy Lie algebra of
    function spaces,  {\em Comment. Math. Helv.} {\bf 83} (2008), 723--739.

\bibitem{BM} U. Buijs and A. Murillo, Algebraic models of non-connected
    spaces and homotopy theory of $L_{\infty}$ algebras, \textit{Adv. Math.}
    {\bf 236} (2013), 60--91.

\bibitem{BFM} U. Buijs, Y. F\'elix and A. Murillo, $L_\infty$ rational
    homotopy theory of mapping spaces, \textit{Rev. Mat. Complut.} {\bf 26}
    (2013), no. 2, 573--588.

\bibitem{BFM11}  U. Buijs, Y. F\' elix and A. Murillo, $L_\infty$ models of
    mapping spaces, {\em J. Math. Soc. Japan} {\bf 63}(2) (2011), 503--524.

\bibitem{BGM} U. Buijs, J.\,J. Guti\'errez and A. Murillo, Derivations, the
    Lawrence--Sullivan interval and the Fiorenza--Manetti mapping cone,
    \textit{Math. Z.} {\bf 273} no. 3-4 (2013), 981--997.

\bibitem{cs} M. Chas and D. Sullivan, String Topology, \textit{preprint} {\tt
    arXiv:math/9911159}

\bibitem{ChG} X.\,Z. Cheng and E. Getzler, Transferring homotopy commutative
    algebraic structures, \textit{J.~Pure Appl. Algebra} {\bf 212} (2008),
    no. 11, 2535--2542.

\bibitem{Cor94} O. Cornea, Cone length and Lusternik-Schnirelmann category,
    \textit{Topology} \textbf{33} (1994), 95--111.

\bibitem{FHT} Y. F\'elix, S. Halperin and J.-C. Thomas, {\em Rational
    Homotopy Theory}, Graduate Texts in Mathematics 205, Spinger (2000).

\bibitem{FT05} Y. F\'elix and D. Tanr\'e, $H$-space structure on pointed
    mapping spaces, \textit{Algebr. Geom. Topol.} {\bf 5} (2005), 713--724.

\bibitem{Fuk03} K. Fukaya, Deformation theory, homological algebra and mirror
    symmetry. \textit{Geometry and physics of branes} (Como, 2001), 121--209,
    Ser. High Energy Phys. Cosmol. Gravit., IOP, Bristol, 2003.

\bibitem{Get09} E. Getzler, Lie theory for nilpotent $L_\infty$ algebras.
    \textit{Ann. of Math} {\bf 170} (2009), no. 1, 271--301.

\bibitem{GS86} V.\,K.\,A.\,M. Gugenheim and J.\,D. Stasheff, On perturbations
    and $A_{\infty}$-structures. \textit{Bull. Soc. Math. Belg. S\'er. A}
    \textbf{38} (1986), 237--246

\bibitem{GLS91}V.\,K.\,A.\,M. Gugenheim, L.\,A. Lambe and J.\,D. Stasheff,
    Perturbation theory in differential homological algebra. II.
    \textit{Illinois J. Math.} \textbf{35} (1991), no. 3, 357--373.

\bibitem{Hae82} A. Haefliger, Rational Homotopy of the space of sections of a
    nilpotent bundle, {\em Trans. Amer. Math. Soc.}, {\bf 273} (1982),
    609--620.

\bibitem{HK91} J. Huebschmann and T. Kadeishvili, Small models for chain
    algebras. \textit{Math. Z.} \textbf{207} (1991), no. 2, 245--280.

\bibitem{Kai83} T. Kadeishvili, The algebraic structure in the homology of an
    $A(\infty)$-algebra. \emph{Soobshch. Akad. Nauk Gruzin. SSR} \textbf{108}
    (1982), no. 2, 249--252 (1983).

\bibitem{Kon03} M. Kontsevich, Deformation quantization of Poisson manifolds,
    \textit{Lett. Math. Phys.}, \textbf{66}(3) (2003), 157--216.

\bibitem{KS00} M. Kontsevich and Y. Soibelman, Deformations of algebras over
    operads and Deligne's conjecture. {\em G. Dito and D. Sternheimer (eds)
    Conf\'erence Mosh\'e Flato 1999 , Vol. I (Dijon 1999), Kluwer Acad.
    Publ., Dordrecht} (2000) 255--307.

\bibitem{KS01} M. Kontsevich and Y. Soibelman, Homological mirror symmetry
    and torus fibrations. \emph{Symplectic geometry and mirror symmetry}
    (Seoul, 2000), 203--263, World Sci. Publ., River Edge, NJ, 2001.

\bibitem{Laz} A. Lazarev, Maurer--Cartan moduli and models for function
    spaces, \textit{Adv. Math.} {\bf 235} (2013), 296--320.

\bibitem{LV} J.-L. Loday and B. Vallette, {\it Algebraic Operads},
    Grundlehren Math. Wiss.,  Vol. 346, Springer-Verlag (2012).

\bibitem{LS10} G. Lupton and S.\,B. Smith, Whitehead products in function
    spaces: Quillen model formulae, \textit{J. Math. Soc. Japan} {\bf 62}
    (2010), no. 1, 49--81.

\bibitem{Mer99} S.\,A. Merkulov, Strong homotopy algebras of a K\"{a}hler
    manifold, \textit{Int. Math. Res. Not.} (1999), vol. 1999, no. 3,
    153--164.

\bibitem{Nei} J. Neisendorfer and T. Miller, Formal and coformal spaces, {\it
    Illinois J. Math.} {\bf 22} (1978), no. 4, 565--580.

\bibitem{SS85} M. Schlessinger and J.\,D. Stasheff, The Lie algebra structure
    of tangent cohomology and deformation theory, \textit{J. Pure Appl.
    Algebra} {\bf 38} (1985), no. 2-3, 313--322.

\bibitem{Sul78} D. Sullivan, Infinitesimal computations in topology,
    \emph{Publ. Math. Inst. Hautes Etudes Sci.} {\bf 47} (1978), 269--331.

\end{thebibliography}
\end{document}